 \documentclass[a4paper,reqno,12pt]{amsart} 
\usepackage{amsmath} 
\usepackage{amssymb}      
\usepackage{yhmath}% http://ctan.org/pkg/yhmath
\usepackage{mathdots}% http://ctan.org/pkg/mathdots
\usepackage{MnSymbol}
\usepackage{amsthm}      
\usepackage{tikz,amsmath}
\usetikzlibrary{arrows}

\usepackage{epsfig}      
\usepackage{graphicx}
\usepackage[normalem]{ulem}
\usepackage[all,line,
]{xy} 
\CompileMatrices %goer at xy-matricerne koerer hurtigere.
   \newcommand{\Aff}{{\operatorname{Aff}}}
 
   \newcommand{\Hom}{\operatorname{Hom}}

 \newcommand{\Imag}{\operatorname{Im}}

\newcommand{\id}{\operatorname{id}} 
\newcommand{\Aut}{\operatorname{Aut}}
\newcommand{\KMS}{\operatorname{KMS}}

 \newcommand{\supp}{\operatorname{supp}}

\newcommand{\ev}{\operatorname{ev}}

   \theoremstyle{plain}%default
   \newtheorem{thm}{Theorem}[section]
   \newtheorem{prop}[thm]{Proposition}
   \newtheorem{lemma}[thm]{Lemma}  
   \newtheorem{cor}[thm]{Corollary}
   \theoremstyle{definition}
   
   \newtheorem{defn}[thm]{Definition}
   
   \theoremstyle{remark}
   \newtheorem{obs}[thm]{Observation}

\newtheorem{poem}[thm]{Additional properties}
\newtheorem{TOM}[thm]{Property}

 \usepackage[stable]{footmisc}

\usepackage{mdframed,xcolor}

\definecolor{mybgcolor}{gray}{0.8}
\definecolor{myframecolor}{rgb}{.647,.129,.149}

\mdfdefinestyle{mystyle}{
  usetwoside=false,
  skipabove=0.6em plus 0.8em minus 0.2em,
  skipbelow=0.6em plus 0.8em minus 0.2em,
  innerleftmargin=.25em,
  innerrightmargin=0.25em,
  innertopmargin=0.25em,
  innerbottommargin=0.25em,
  leftmargin=-.75em,
  rightmargin=-0em,
  topline=false,
  rightline=false,
  bottomline=false,
  leftline=false,
  backgroundcolor=mybgcolor,
  splittopskip=0.75em,
  splitbottomskip=0.25em,
  innerleftmargin=0.5em,
  leftline=true,
  linecolor=myframecolor,
  linewidth=0.25em,
}

\newmdenv[style=mystyle]{important}

\usepackage{lipsum}

   \numberwithin{equation}{section}

        \date{\today}
\title[$\KMS_\infty$ states]{Equilibria when the temperature goes to zero}
\author{Klaus Thomsen}
\date{\today}

\bigskip

\bigskip

\address{Department of Mathematics, Aarhus University, Ny Munkegade, 8000 Aarhus C, Denmark}

\email{matkt@math.au.dk}

\begin{document}

\begin{abstract} We present methods to construct flows with varying set of $\KMS_\infty$ states on a given simple unital AF-algebra. It follows, for example, that for any pair $D_+$ and $D_-$ of non-empty compact metric spaces there is a flow $\sigma = (\sigma_t)_{t \in \mathbb R}$ on the CAR algebra whose set of $\KMS_\infty$ states is homeomorphic to $D_+$ while the set of $\KMS_{\infty}$ states for the inverted flow $(\sigma_{-t})_{t \in \mathbb R}$ is homeomorphic to $D_-$. Remarkably the flows that realize all such pairs $D_\pm$ can be chosen to have isomorphic $\KMS$ bundles. 

\end{abstract}

\maketitle

\section{Introduction}

In certain quantum statistical models where the observables are represented by elements of a unital $C^*$-algebra $A$ and the time-evolution is given by a flow $\sigma = (\sigma_t)_{t \in \mathbb R}$ of automorphisms of $A$ there is a notion of ground states, \cite{BR}; states that may be interpreted as equilibrium states at zero temperature. Inside the set of ground states is a subset consisting of the states that are limits in the weak* topology of states that are equilibrium states for the system when the temperature goes to zero. These ground states were introduced by Connes and Marcolli in \cite{CM1} and they constitute often a considerable smaller and better behaved set of ground states. They are called KMS infinity states, in the following abbreviated to $\KMS_\infty$ states. We refer to \cite{CM1}, \cite{CM2}, \cite{CMR}, \cite{LLN1}, \cite{aHLRS}, \cite{CL}, \cite{CT}, \cite{LLN2}, \cite{BLT}, \cite{Th2}, \cite{NS} and \cite{ALMS} for papers where $\KMS_\infty$ states are considered. The equilibrium states of the models at a finite inverse temperature $\beta$ are called $\beta$-KMS states and they form a Choquet simplex $S^\sigma_\beta$ for each $\beta \in  \mathbb R$. When we denote the set of $\KMS_\infty$ states of $\sigma$ by $\KMS_\infty(\sigma)$, the defining relation between $\KMS_\infty(\sigma)$ and the sets of $\KMS$ states is that
$$
\KMS_\infty(\sigma) := \bigcap_{R > 0} \overline{\bigcup_{\beta \geq R} S^\sigma_\beta} 
$$
where $\overline{\bigcup_{\beta \geq R} S^\sigma_\beta}$ is the closure of $\bigcup_{\beta \geq R} S^\sigma_\beta$ in the weak* topology. With this definition we follow \cite{LLN1} and the subsequent papers, but it must be emphasized that this is not the original definition of Connes and Marcolli. In the first three papers \cite{CM1}, \cite{CM2} and \cite{CMR} mentioned above, a state $\omega$ is a $\KMS_\infty$ state when there are $\beta$-KMS states $\omega_\beta$ for all sufficiently large $\beta$ such that $\lim_{\beta \to \infty} \omega_\beta = \omega$, while we only require that there is a sequence $\{\omega_n\}$ of states such that $\omega_n$ is a $\beta_n$-KMS state, $\lim_{n \to \infty} \beta_n = \infty$ and $\lim_{n \to \infty} \omega_n = \omega$. The two conditions are very different. For example there may be many states satisfying the second condition, but none satisfying the first. In a final remark of the paper we say a little more about the difference.

  Once the set of $\KMS_\infty$ states is considered it is also natural to consider the set of $\KMS_{-\infty}$ states which are the states that are limits in the weak* topology of KMS states for the system when the inverse temperature goes to minus  infinity; in symbols
$$
\KMS_{-\infty}(\sigma) := \bigcap_{R < 0} \overline{\bigcup_{\beta \leq R} S^\sigma_\beta} .
$$
This is the set of $\KMS_\infty$ states for the inverted system where the time-parameter $t$ is replaced by $-t$. We are here interested both in the set of $\KMS_\infty$ states and the set of $\KMS_{-\infty}$ states for flows on unital $C^*$-algebras. Since the state space of a unital $C^*$-algebra is compact in the weak* topology, $\KMS_\infty(\sigma)$ is a compact space, and metric when the $C^*$-algebra is separable as it always will be in this paper. Hence $\KMS_\infty(\sigma)$ and $\KMS_{-\infty}(\sigma)$ are both compact metric spaces. We would like to know which compact metric spaces can be realized as $\KMS_\infty(\sigma)$ and $\KMS_{-\infty}(\sigma)$ for a flow on a $C^*$-algebra, and how these sets are related to the collection of KMS states and the $C^*$-algebra of observables. With the results we obtain here this goal is achieved when the $C^*$-algebra is a simple infinite dimensional unital AF-algebra.

There are certain phenomena that can occur for general flows, but never for non-trivial flows on a simple unital $C^*$-algebra which we consider to be the most interesting. In order to focus on the interesting cases we introduce therefore here the notion of \emph{a simple flow} by which we mean a non-trivial flow $\sigma$ on a unital separable $C^*$-algebra $B$ such that no non-trivial closed two-sided ideal of $B$ is left globally invariant by $\sigma$. For simple flows we show in Section \ref{bundle} that \emph{the KMS bundle}, which is a notion introduced and exploited in \cite{ET}, can be identified with the union 
$$
S^\sigma := \bigcup_{\beta \in \mathbb R} S^\sigma_\beta
$$
equipped with the weak* topology inherited from the state space, together with the map $\pi^\sigma : S^\sigma \to \mathbb R$ defined such that $\pi^\sigma(S^\sigma_\beta) = \{\beta\}$.\footnote{In general, when the flow is not a simple flow, some intersections $S^\sigma_\beta \cap S^{\sigma}_{\beta'}$ with $\beta \neq \beta'$ may be non-empty in which case the map $\pi^\sigma$ is not defined.} Furthermore, for simple flows the map $\pi^\sigma$ extends to a continuous map
$$
\pi^\sigma : \ \KMS_{-\infty}(\sigma) \cup  \left(\bigcup_{\beta \in \mathbb R} S^\sigma_\beta\right) \cup \KMS_{\infty}(\sigma) \ \ \to \ \ [-\infty,\infty]
$$
such that ${\pi^{\sigma}}^{-1}(-\infty) = \KMS_{-\infty}(\sigma) $ and ${\pi^{\sigma}}^{-1}(\infty) = \KMS_{\infty}(\sigma)$. It follows in particular that the topological boundary $\overline{S^\sigma}\backslash S^\sigma$ of $S^\sigma$ inside the state space is the disjoint union of $\KMS_\infty(\sigma)$ and $\KMS_{-\infty}(\sigma)$. In symbols,
  $$
  \overline{S^\sigma}\backslash S^\sigma = \KMS_{-\infty}(\sigma) \sqcup \KMS_\infty(\sigma) .
  $$

  The ingredients in the problem we consider can be illustrated by the picture
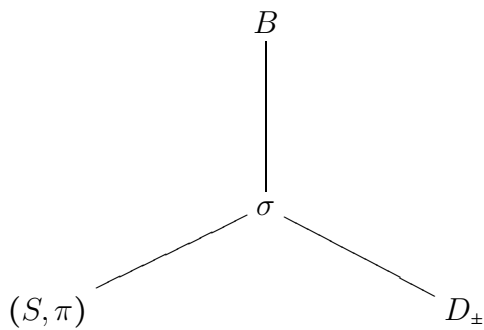
\begin{equation*}
\begin{xymatrix}{
& & &&  B &&&& \\
&&&&&&&&\\
& & && \sigma \ar@{-}[uu] &&&& \\
& &(S,\pi) \ar@{-}[urr]   &&  && D_\pm \ar@{-}[ull]& &\\
}
\end{xymatrix}
\end{equation*}
where $B$ is a unital separable $C^*$-algebra, $(S,\pi)$ is a proper simplex bundle, and $D_+$ and $D_-$ are compact metric spaces. The $\sigma$ in the middle is a simple flow which binds the ingredients together in the sense that $\sigma$ acts on $B$ and the $\KMS$ bundle of $\sigma$ is isomorphic to $(S,\pi)$ while $\KMS_\infty(\sigma)$ is homeomorphic to $D_+$ and $\KMS_{-\infty}(\sigma)$ is homeomorphic to $D_-$. The problem we consider is to decide for a given quadruple $B,(S,\pi), D_+$ and $D_-$, if there is a flow $\sigma$ fitting into such a picture. Concerning the left-hand side, involving only $B,(S,\pi)$ and $\sigma$, quite a lot is known by now. When we ignore the spaces $D_\pm$ like this, there is only one rather obvious obstruction to the existence of $\sigma$: Since the $0$-KMS states for $\sigma$ are the $\sigma$-invariant trace states there has to be a relation between the tracial state space of $B$ and the simplex $\pi^{-1}(0)$. In many cases this means that there can only be a flow on $B$ whose KMS bundle is isomorphic to $(S,\pi)$ when the tracial state space of $B$ is affinely homeomorphic to $\pi^{-1}(0)$. It may be that this is the only obstruction when $B$ is a unital simple separable and infinite dimensional $C^*$-algebra. See \cite{ET}, \cite{EST} and \cite{ES} for results in that direction. 

Now assume that the left-hand side of the picture can be realized. That is, assume that $B$ is a simple unital separable $C^*$-algebra with a flow whose $\KMS$ bundle is isomorphic to $(S,\pi)$. Consider then two compact metric spaces $D_+$ and $D_-$. Is there a flow $\sigma$ completing the picture above? As before there is one obvious obstruction: If the set $\pi(S)$ is bounded above, which means that there is a lower bound on the positive temperatures realized by any flow whose $\KMS$ bundle is isomorphic to $(S,\pi)$, then the space $D_+$ must be empty. Likewise, if $\pi(S)$ is bounded below the set $D_-$ must be empty. In general there is one more, perhaps less obvious restriction as we point out in Lemma \ref{05-08-23}: $D_+$ must be a continuous image of the Stone-\v{C}ech boundary of $\pi^{-1}([0,\infty))$ and $D_-$ must be a continuous image of the Stone-\v{C}ech boundary of $\pi^{-1}((-\infty,0])$. These conditions imply that when $\pi(S)$ contains an infinite interval $[r,\infty)$ for some $r > 0$, the space $D_+$ has to be connected. And similarly for $D_-$. The main result we obtain is that there are no more restrictions when $B$ is a unital simple infinite dimensional AF-algebra. That is, if $B$ is such an algebra and the left-hand of the diagram above is given with $\pi(S)$ unbounded in both directions, we can complete the diagram with any pair of non-empty compact metric spaces $D_\pm$ provided only that $D_+$ is connected if $\pi(S)$ contains an interval of the form $[r,\infty)$ for some $r > 0$ and that $D_-$ is connected if $-\pi(S)$ contains such an interval. But the flow $\sigma$ will typically have to change of course.

%\begin{thm}\label{mainII} Let $B$ be a simple infinite dimensional unital AF-algebra and let $(S,\pi)$ be a proper simplex bundle such that $\pi^{-1}(0)$ is affinely homeomorphic to the tracial state space of $B$. Assume that $\pi(S)$ is disconnected both at infinity and at minus infinity. Let $D_+$ and $D_-$ be arbitrary non-empty compact metric spaces.

%There is a $2\pi$-periodic flow $\sigma$ on $B$ such that the $\KMS$ bundle of $\sigma$ is isomorphic to $(S,\pi)$, the set of $\KMS_{\infty}$ states of $\sigma$ is homeomorphic to $D_+$ and the set of $\KMS_{-\infty}$ states of $\sigma$ is homeomorphic to $D_-$.
%\end{thm}   

%In particular, it follows that $\KMS_\infty(\sigma)$ in general is not a convex set of states.

\subsection{On the methods} The first examples of flows with a rich structure of KMS states were given in the work of Bratteli, Elliott and Herman in \cite{BEH}. To describe the basic idea in their work, recall that a crossed product $A \rtimes_\alpha \mathbb Z$ of a $C^*$-algebra $A$ by an action of $\mathbb Z$ given by an automorphism $\alpha$ is generated by products of the form $au^k$, where $ a  \in A, \ k \in \mathbb Z$, and $u$ is a unitary $u$ with the property that $uau^* = \alpha(a)$ for all $a \in A$. There is a flow $\hat{\alpha} = (\hat{\alpha}_t)_{t \in \mathbb R}$ on $A \rtimes_\alpha \mathbb Z$, which in the following is called \emph{the dual flow} of $\alpha$, defined such that 
$$
\hat{\alpha}_t(au^k) = \exp(ikt)au^k
$$
for all $a \in A, \ k \in \mathbb Z$ and $t \in \mathbb R$. Let $e$ be a projection in $A$. Then $\hat{\alpha}$ fixes $e$ and the corner $e(A\rtimes_\alpha \mathbb Z)e$ is $\hat{\alpha}$-invariant so that the dual flow defines a flow $\hat{\alpha}^e$ on $e(A\rtimes_\alpha \mathbb Z)e$ by restriction. Note that $e(A\rtimes_\alpha \mathbb Z)e$ is unital even when $A$ isn't, and simple when $A \rtimes_\alpha \mathbb Z$ is. The idea in \cite{BEH} was to construct $A,\alpha$ and $e$ such that $\hat{\alpha}^e$ has a rich structure of KMS states; in \cite{BEH} the focus was on the set of $\beta\in \mathbb R $ for which there is a $\beta$-KMS state and the main result was the construction, for any given closed subset $F$ of $\mathbb R$, of a triple $A,\alpha, e$ such that there is a $\beta$-KMS state for the resulting flow $\hat{\alpha}^e$ if and only if $\beta \in F$. What made this approach possible was the classification of AF-algebras that had just been completed around 1979.

The basic idea from \cite{BEH} was continued in \cite{BEK1}, \cite{BEK2}, and taken up again in \cite{Th3}, \cite{ET} and \cite{EST} where it was combined with recent classification results for simple $C^*$-algebras that go beyond the AF-case. The construction we present here uses methods from \cite{ET} where they were used to find a flow on a given AF-algebra $B$ with a $\KMS$ bundle isomorphic to a given proper simplex bundle $(S,\pi)$. These methods are here combined with ideas from \cite{Th3} to control the asymptotic behaviour of the sets of $\KMS$ states when the inverse temperature goes to plus or minus infinity. As in \cite{ET} the crux of the matter is to construct an AF-algebra $A$, an automorphism $\alpha$ of $A$ and a projection $e \in A$ such that 
\begin{itemize}
\item[(a)] the restriction $\hat{\alpha}^e$ of the dual flow $\hat{\alpha}$ to $e(A \rtimes_\alpha \mathbb Z)e$ has a $\KMS$ bundle isomorphic to $(S,\pi)$, and
\item[(b)] the corner $e(A \rtimes_\alpha \mathbb Z)e$ is isomorphic $B$.
\end{itemize}
This time we also want to arrange that
\begin{itemize}
\item[(c)] the set of $\KMS_{\infty}$ states of $\hat{\alpha}^e$ is homeomorphic to $D_+$ and the set of $\KMS_{-\infty}$ states of $\hat{\alpha}$ is homeomorphic to $D_-$
\end{itemize}
The desired flow $\sigma$ on $B$ is then obtained by transferring $\hat{\alpha}^e$ to $B$ using the isomorphism from (b).% The possibility of using results from the classification of simple unital $C^*$-algebras to arrange (a) was first demonstrated in \cite{Th2}. The more fundamental idea to use the classification of AF-algebras to construct $\hat{\alpha}^e$ such that it has a rich structure of $\KMS$ states goes back to the work by Bratteli, Elliott and Herman, \cite{BEH}. 

 \section{On $\KMS$ bundles}\label{bundle}

%\section{On the $\KMS$-states bundle}\footnote{Readers already familiar with proper simplex bundles can skip this section until we discuss the main results in the final section.}

Let $\sigma$ be a flow on a unital $C^*$-algebra $A$ and let $\mathcal A_\sigma$ denote the dense $*$-algebra in $A$ consisting of the elements that are analytic for $\sigma$. Let $\beta \in \mathbb R$. A state $\omega$ on $A$ is a \emph{$\beta$-KMS state} for $\sigma$ when $\omega$ is $\sigma$-invariant and
$$
\omega(ab) = \omega(b \sigma_{i\beta}(a))
$$
for all $b \in A$ and all elements $a\in \mathcal A_\sigma$, \cite{BR}. The set of $\beta$-KMS states for $\sigma$ is a possibly empty Choquet simplex which we denote by $S^\sigma_\beta$ in the following. The general structure of the collection 
$$
\bigcup_{\beta \in \mathbb R} S^\sigma_\beta
$$ 
of all KMS states for a flow is best described using the notion of \emph{proper simplex bundles} which we describe next, cf. \cite{ET}. Let $S$ be a second countable locally compact Hausdorff space and $\pi : S \to \mathbb R$ a continuous map. If the inverse image $\pi^{-1}(t)$, equipped with the relative topology inherited from $S$, is homeomorphic to a compact Choquet simplex for all $t \in \mathbb R$ we say that $(S,\pi)$ is a \emph{simplex bundle}. We emphasize that $\pi$ need not be surjective, and we consider therefore also the empty set as a Choquet simplex. When $(S,\pi)$ is a simplex bundle we denote by $\mathcal A(S,\pi)$ the set of continuous functions $f : S \to \mathbb R$ with the property that the restriction $f|_{\pi^{-1}(t)}$ of $f$ to $\pi^{-1}(t)$ is affine for all $t \in \mathbb R$.

\begin{defn}\label{25-08-21} A simplex bundle $(S,\pi)$ is a \emph{proper simplex bundle} when
\begin{itemize}
\item[(1)] $\pi$ is proper, i.e. $\pi^{-1}(K)$ is compact in $S$ when $K \subseteq \mathbb R$ is compact, and
\item[(2)] $\mathcal A(S,\pi)$ separates points on $S$; i.e. for all $x\neq y$ in $S$ there is an $f \in\mathcal A(S,\pi)$ such that $f(x) \neq f(y)$.
\end{itemize}
\end{defn}

Two proper simplex bundles $(S,\pi)$ and $(S',\pi')$ are \emph{isomorphic} when there is a homeomorphism $\phi : S \to S'$ such that $\pi' \circ \phi = \pi$ and $\phi: \pi^{-1}(\beta) \to {\pi'}^{-1}(\beta)$ is affine for all $\beta \in \mathbb R$.

Let $E(A)$ denote the state space of $A$ which is a compact topological space in the weak* topology. Set 
$$
S^\sigma := \left\{(\omega,\beta) \in E(A)\times \mathbb R: \ \omega \in S^\sigma_\beta \right\} .
$$
We consider $S^\sigma$ as a topological space in the topology inherited from the product topology of $E(A) \times \mathbb R$. The projection $\pi^\sigma : S^\sigma \to \mathbb R$ to the second coordinate is then continuous. We call $(S^\sigma,\pi^\sigma)$ the \emph{KMS bundle of $\sigma$}. The following is a consequence of results in \cite{ET} and \cite{EST} and shows the significance of proper simplex bundles.

\begin{prop}\label{12-11-22} Let $\sigma$ be a flow on a unital separable $C^*$-algebra. Then $(S^\sigma,\pi^\sigma)$ is a proper simplex bundle, and every proper simplex bundle is isomorphic to the $\KMS$ bundle of a flow on a simple unital separable $C^*$-algebra.
\end{prop}

 We say that $A$ is \emph{$\sigma$-simple} when the only closed two-sided $\sigma$-invariant ideals in $A$ are $\{0\}$ and $A$, and that $\sigma$ is \emph{non-trivial} when $\sigma$ is not the trivial flow; the one for which $\sigma_t = \id_A$ for all $t \in \mathbb R$. A $\sigma$-simple non-trivial flow on a unital separable $C^*$-algebra will be called a \emph{simple flow}. In particular, every non-trivial flow on a simple unital separable $C^*$-algebra is a simple flow.

\begin{lemma}\label{05-01-23}  Assume $\sigma$ is a simple flow. Then $S^\sigma_\beta \cap S^\sigma_{\beta'} =\emptyset$ when $\beta \neq \beta'$.
\end{lemma}

\begin{proof} Assume $\omega \in S^\sigma_\beta \cap S^\sigma_{\beta'}$ and let $(H_\omega,\pi_\omega, \Omega_\omega)$ be the GNS representation of $\omega$, and $\widehat{\omega}$ the normal extension of $\omega$ to $\pi_\omega(A)''$. Then $\widehat{\omega} \circ \pi_\omega = \omega$. By Corollary 5.3.4 in \cite{BR} there is a $\sigma$-weakly continuous flow $\widehat{\sigma}$ on $\pi_\omega(A)''$ for which $\widehat{\omega}$ is both a $\beta$- and a $\beta'$-KMS state. Note that $\widehat{\sigma}_t \circ \pi_\omega = \pi_{\omega} \circ \sigma_t$ for all $t$ and that $\widehat{\omega}$ is faithful by Corollary 5.3.9 of \cite{BR}. It follows from Proposition 5.3.7 in \cite{BR} that $\{\widehat{\sigma}_{- \beta t}\}$ and $\{\widehat{\sigma}_{-\beta' t}\}$ are both modular flows for $\widehat{\omega}$ and hence by the uniqueness of the modular flow, cf. e.g. Theorem 9.2.16 in \cite{KR}, that $\widehat{\sigma}_{- \beta t} = \widehat{\sigma}_{-\beta' t}$ for all $t$. Hence $\pi_\omega\left( \sigma_{-\beta t}(a) - \sigma_{-\beta' t}(a)\right) = 0$ for all $t,a$. Note that $\ker \pi_\omega$ is a proper $\sigma$-invariant ideal and hence $\{0\}$ by assumption. It follows therefore first that $\sigma_{\beta t} = \sigma_{\beta' t}$ for all $t \in \mathbb R$, and then since $\beta \neq \beta'$ that $\sigma$ is trivial, contrary to assumption. It follows that $S^\sigma_\beta \cap S^\sigma_{\beta'} =\emptyset$.
\end{proof}

When $\sigma$ is simple it follows from Lemma \ref{05-01-23} that the map 
\begin{equation}\label{31-12-22a}
S^\sigma \ni (\omega,\beta) \mapsto \omega \in \bigcup_{\beta \in \mathbb R} S^\sigma_\beta
\end{equation}
is injective and hence a bijection. Thus $\bigcup_{\beta \in \mathbb R} S^\sigma_\beta$ has a topology in which it is a second countable locally compact Hausdorff space since $S^\sigma$ does. To identify this topology, let 
\begin{equation}\label{01-01-23a}
\Phi: \
\bigcup_{\beta \in \mathbb R} S^\sigma_\beta \to \mathbb R
\end{equation}
be the function defined such that $\Phi(\omega) = \beta$ when $\omega \in S^\sigma_\beta$. This is well-defined by Lemma \ref{05-01-23}.

 \begin{lemma}\label{04-01-23bx} Assume that $A$ is unital and separable, and that $\omega$ is a $\KMS_\infty$ state for $\sigma$. Then 
\begin{equation}\label{04-01-23d}
\sup_{\Imag z \geq 0} \left|\omega(a\sigma_z(b))\right| \leq \|a\|\|b\|
\end{equation}
for all $a \in A, \ b \in \mathcal A_\sigma$.
\end{lemma} 
\begin{proof} Let $\{\omega_n\}$ be a sequence of states such that $\lim_{n \to \infty} \omega_n = \omega$ in the weak* topology, $\omega_n$ is a $\beta_n$-KMS state and $\lim_{n \to \infty} \beta_n = \infty$. Then $\lim_{n \to \infty} \omega_n(a\sigma_z(b)) = \omega(a\sigma_z(b))$ for all $z \in \mathbb C$ and to establish \eqref{04-01-23d} it suffices to show that
\begin{equation*}\label{01-01-23d}
\left| \omega_n(a\sigma_z(b))\right| \leq \|a\|\|b\| 
\end{equation*}
when $0 \leq \Imag z \leq \beta_n$, which is part of Proposition 5.3.7 of \cite{BR}.

\end{proof}

\begin{lemma}\label{05-01-23a} Let $\sigma$ be a simple flow. Then $S^\sigma_\beta \cap \KMS_\infty(\sigma)= \emptyset$ for all $\beta \in \mathbb R$.
\end{lemma}
\begin{proof} Assume for a contradiction that there is an element $\omega$ in $S^\sigma_\beta \cap \KMS_\infty(\sigma)$.  Let $a,b \in \mathcal A_\sigma$. By Lemma \ref{04-01-23bx} the function $z \mapsto \omega(a\sigma_z(b))$ is bounded for $\Imag z \geq 0$. Since $\omega$ is a $\beta$-KMS state we have that 
$$
\omega(\sigma_{-i\beta}(a)a^* \sigma_z(b)) = \omega(a^*\sigma_z(b)a)
$$
for all $z \in \mathbb C$. In particular,
$$
\omega(\sigma_{-i\beta}(a)a^* \sigma_t(b)) = \omega(a^*\sigma_t(b)a) \in \mathbb R
$$
for $t \in \mathbb R$ when $b=b^* \in \mathcal A_\sigma$. It follows therefore from the Schwarz reflection principle, cf. e.g. Theorem 11.14 in \cite{Ru}, that there is an entire function $F$ such that $F(z) = \omega(\sigma_{-i\beta}(a)a^* \sigma_z(b))$ when $\Imag z \geq 0$ and $F(z) = \overline{F(\overline{z})}$ when $\Imag z \leq 0$. This function is bounded on $\mathbb C$ since $\omega(\sigma_{-i\beta}(a)a^* \sigma_z(b))$ is bounded for $\Imag z \geq 0$, and hence constant by Liouville's theorem. Since $z \mapsto \omega(\sigma_{-i\beta}(a)a^* \sigma_z(b))$ is also entire and agrees with $F$ when $\Imag z \geq 0$, it must be equal to $F$ and hence be constant. Thus
\begin{equation}\label{20-07-23}
\omega(\sigma_{-i\beta}(a)a^* \sigma_z(b)) =\omega(\sigma_{-i\beta}(a)a^* b)
\end{equation}
for all $z \in \mathbb C$ when $a,b \in \mathcal A_\sigma$ and $b=b^*$. Since $\mathcal A_\sigma$ is a $*$-algebra every element of $\mathcal A_\sigma$ is a linear combination of two self-adjoint elements from $\mathcal A_\sigma$. It follows therefore that \eqref{20-07-23} holds for all $a,b \in \mathcal A_\sigma$ and all $z \in \mathbb C$. Using the polarisation identity
$$
\sigma_{-i\beta}(x)y^* = \frac{1}{4} \sum_{k=1}^4 i^k \sigma_{-i \beta}(x+i^ky) (x+i^ky)^*
$$
we find that
$$
\omega(\sigma_{-i\beta}(a)b^* \sigma_z(c)) =\omega(\sigma_{-i\beta}(a)b^* c)
$$
when $z \in \mathbb C$ and $a,b,c \in \mathcal A_\sigma$. Taking $a =1$ we find that $\omega(b\sigma_z(c)) = \omega(bc)$ for all  $z \in \mathbb C$ and $b,c \in \mathcal A_\sigma$. Since $\omega$ is a $\beta$-KMS state we have that $\omega(cb) = \omega( b \sigma_{i\beta}(c)) = \omega(bc)$ for all $b,c \in \mathcal A_\sigma$, and it follows therefore that $\omega$ is a $t$-KMS state for all $t \in \mathbb R$. Since $\sigma$ is a simple flow this is impossible by Lemma \ref{05-01-23}. 
\end{proof}

\begin{lemma}\label{01-01-23} Let $\sigma$ be a simple flow. The map $\Phi$ of \eqref{01-01-23a} is continuous with respect to the weak* topology.
\end{lemma}
\begin{proof} Let $\{\omega_n\}_{n =1}^\infty$ and $\omega$ be elements of $\bigcup_{\beta \in \mathbb R} S^\sigma_\beta $ and assume that $\lim_{n\to \infty} \omega_n = \omega$ in the weak* topology. Set $\beta := \Phi(\omega)$ and $\beta_n := \Phi(\omega_n)$. If $\{\beta_n\}$ is not bounded it follows that $\omega$ is either a $\KMS_\infty$ state or a $\KMS_{-\infty}$-state, and both possibilities are impossible by Lemma \ref{05-01-23a} since $\omega$ is a $\beta$-KMS state for $\sigma$ and a $-\beta$-KMS state for the inverted flow $\sigma_{-t}$. Hence $\sup_n|\beta_n| < \infty$. If $\{\beta_n\}$ does not convergence to $\beta$ there is a $\beta' \neq \beta$ and a subsequence $\{\beta_{n_k}\}$ such that $\lim_{k \to \infty} \beta_{n_k} = \beta'$. Then $\omega$ is a $\beta'$-KMS state for $\sigma$ by Proposition 5.3.23 of \cite{BR} which is impossible by Lemma \ref{05-01-23}. It follows that $\lim_{n \to \infty} \beta_n = \beta$.
\end{proof}

\begin{prop}\label{01-01-23e} Let $\sigma$ be a simple flow. Then $\bigcup_{\beta \in \mathbb R} S^\sigma_\beta$ is a second countable locally compact Hausdorff space in the weak* topology and $$
\left(\bigcup_{\beta \in \mathbb R} S^\sigma_\beta ,\Phi\right)
$$ 
is a proper simplex bundle isomorphic to the $\KMS$ bundle of $\sigma$.
\end{prop}
\begin{proof} Thanks to Lemma \ref{01-01-23} we can define a continuous map 
$$
\Psi : \bigcup_{\beta \in \mathbb R} S^\sigma_\beta  \to S^\sigma
$$
by $\Psi(\omega) = (\omega, \Phi(\omega))$. This is clearly the inverse of the map \eqref{31-12-22a}.
\end{proof}

%A state $\omega \in E(A)$ is a $\KMS_\infty$ (resp. $\KMS_{-\infty}$) state for $\sigma$ when there is a sequence $\{\omega_n\}$ in $E(A)$ such that $\omega_n$ is a $\beta_n$-KMS state, $\lim_{n \to \infty} \omega_n = \omega$ in the weak* topology and $\lim_{n \to \infty} \beta_n = \infty$ (resp. $\lim_{n \to \infty} \beta_n = - \infty$). We denote the set of $\KMS_{\infty}$ states by $\KMS_\infty(\sigma)$ and the set of $\KMS_{-\infty}$ states by $\KMS_{-\infty}(\sigma)$.

\begin{lemma}\label{04-01-23a} Assume that $A$ is unital and let $\omega \in \KMS_\infty(\sigma) \cap \KMS_{-\infty}(\sigma)$. There is a $\sigma$-invariant proper ideal $I \subseteq A$ such that $\sigma$ induces the trivial flow on $A/I$.
\end{lemma}
\begin{proof} Since $\omega$ is a $\KMS_\infty$ state for the inverted flow $\sigma_{-t}$ it follows from Lemma \ref{04-01-23bx} that $\sup_{z \in \mathbb C} \left|\omega(a\sigma_{z}(b))\right| \leq \|a\|\|b\|$ and then by Liouville's theorem that $\omega(a\sigma_z(b)) = \omega(ab)$ for all $z \in \mathbb C$ when $a \in A, \ b \in \mathcal A_\sigma$. Let $(H_\omega,\Omega_\omega,\pi_\omega)$ be the GNS representation of $A$ coming from $\omega$. As a limit of $\sigma$-invariant states $\omega$ is itself $\sigma$-invariant and there is therefore a densely defined self-adjoint operator $D_\omega$ on $H_\omega$ such that $\pi_\omega(\mathcal A_\sigma)\Omega_\omega$ is a core for $D_\omega$ and 
$$
e^{itD_\omega}\pi_\omega(a)\Omega_\omega = \pi_\omega(\sigma_t(a))\Omega_\omega
$$
when $t \in \mathbb R$ and $a \in \mathcal A_\sigma$, cf. \cite{BR}; in particular Corollary 3.1.7 of \cite{BR}. Since $t \mapsto \omega(a\sigma_t(b))$ is constant we find that
\begin{align*}
&\left<  D_\omega\pi_\omega(b)\Omega_\omega, \pi_{\omega}(a^*)\Omega_\omega \right> =  -i \frac{\mathrm d}{\mathrm d t} \left<  e^{it D_\omega} \pi_\omega(b) \Omega_\omega ,  \pi_\omega(a^*) \Omega_\omega\right>|_{t =0} \\
& =-i \frac{\mathrm d}{\mathrm d t} \left<  \pi_\omega(\sigma_t(b)) \Omega_\omega, \pi_\omega(a^*) \Omega_\omega \right>|_{t =0}  =-i \frac{\mathrm d}{\mathrm d t} \omega(a\sigma_t(b))|_{t =0} = 0\\
\end{align*}
for all $a \in A, \ b \in \mathcal A_\sigma$, implying that $D_\omega =0$. Hence
$$
\pi_\omega(\sigma_t(a)) = e^{it D_\omega} \pi_{\omega}(a)e^{-it D_\omega} = \pi_\omega(a)
$$
for all $a \in A$. Then $I := \ker \pi_\omega$ is an ideal with the stated properties. 
\end{proof}

\begin{cor}\label{04-01-23}  Let $\sigma$ be a simple flow. Then $\KMS_\infty(\sigma) \cap  \KMS_{-\infty}(\sigma)= \emptyset$.
\end{cor}
\begin{proof} Since $\sigma$ is simple there can not be an ideal $I$ in $A$ with the properties specified in Lemma \ref{04-01-23a}.
\end{proof}

\begin{prop}\label{05-01-23b} Let $\sigma$ be a simple flow. Let $S$ be the closure in the weak* topology of the set $\bigcup_{\beta \in \mathbb R} S^\sigma_\beta$ of all $\KMS$ states for $\sigma$. There is a continuous map $\Phi : S \to [-\infty, \infty]$ such that $\Phi^{-1}(\pm \infty) = \KMS_{\pm \infty}(\sigma)$ and such that
$$
(\Phi^{-1}(\mathbb R),\Phi)
$$
is a proper simplex bundle isomorphic to the $\KMS$ bundle for $\sigma$.
\end{prop} 
\begin{proof} This follows from Proposition \ref{01-01-23e}, Lemma \ref{05-01-23a} and Corollary  \ref{04-01-23}. 
\end{proof}

\subsection{Consequences for the sets $\KMS_{\pm \infty}(\sigma)$.}

 Given a locally compact Hausdorff space $X$ we denote by $\beta X$ the Stone-\v{C}ech compactification of $X$ and by
 $$
 \partial X := \beta X \backslash X
 $$
its Stone-\v{C}ech boundary, or remainder as it is often called. There are many ways to define the Stone-\v{C}ech compactification of $X$; here we view it as the Gelfand spectrum,  a.k.a. the maximal ideal space or the character space, of the $C^*$-algebra $C_b(X)$ of continuous bounded functions on $X$.

 Given a proper simplex bundle $(S,\pi)$, set $S_+ := \pi^{-1}([0,\infty))$ and $S_- := \pi^{-1}((-\infty,0])$.

\begin{lemma}\label{05-08-23} Let $(S,\pi)$ be a proper simplex bundle and $\sigma$ a simple flow such that $(S,\pi)$ is isomorphic to $(S^\sigma,\pi^\sigma)$. There are continuous surjective maps $\chi_+ : \partial S_+ \to \KMS_{\infty}(\sigma)$ and $\chi_- : \partial S_- \to \KMS_{-\infty}(\sigma)$.
\end{lemma}
\begin{proof} Let $\chi : S \to S^\sigma$ be a homeomorphism such that $\pi^\sigma \circ \chi = \pi$. Then $\chi$ restricts to a homeomorphism $\chi : S_+ \to S^\sigma_+$ and we define a $*$-homomorphism $\chi^* : C (\overline{ S^\sigma_+ }) \to C_{b}(S_+)$ by $\chi^*(h) = h\circ \chi$. Note that $\chi^*$ is injective since $S^\sigma_+ = \chi(S_+)$ is dense in $\overline{S^\sigma_+}$. It follows that there is a continuous surjection $\chi_+ : \beta S_+ \to \overline{S^\sigma_+}$ such that $\chi^*(h)(\xi)= h(\chi_+(\xi))$ for all $\xi \in \beta S_+$ when we make the identification $  C_{b}(S_+) = C(\beta S_+)$. Note that $\chi_+|_{S_+} = \chi|_{S_+}$. We claim that $\chi_+(\partial S_+) = \KMS_\infty(\sigma)$. To see this let $\eta \in \KMS_\infty(\sigma)$. Then $\eta = \chi_+(\xi)$ for some $\xi \in \beta S_+$. If $\xi \in S_+$ we have that $ g(\eta) = g(\chi_+(\xi)) = \chi^*(g)(\xi) =g(\chi(\xi))$ for all $g \in C(\overline{S^\sigma_+})$ implying that $\eta = \chi(\xi)$. But $\pi^\sigma(\chi(\xi)) = \pi(\xi) \in [0,\infty)$ which means that $\eta \in \KMS_\infty(\sigma) \cap S^\sigma_{\pi(\xi)}$, contradicting Lemma \ref{05-01-23a}. Hence $\xi \in \partial S_+$ and we conclude that $\KMS_\infty(\sigma) \subseteq \chi_+(\partial S_+)$. The converse inclusion is easy. The map $\chi_-$ is constructed in the same way.
\end{proof}

\begin{cor}\label{06-08-23a}  Let $(S,\pi)$ be a proper simplex bundle and $\sigma$ a simple flow such that $(S,\pi)$ is isomorphic to $(S^\sigma,\pi^\sigma)$. If $\pi(S)$ contains an interval of the form $[r,\infty)$ the space $\KMS_\infty(\sigma)$ is a connected compact metric space.  
\end{cor}
\begin{proof} If $r > 0$ and $[r,\infty) \subseteq \pi(S)$, the set $\pi^{-1}([r,\infty))$ is connected. To see this, assume $ \pi^{-1}([r,\infty)) =A \cup B$ where $A$ and $B$ are closed in $S$, non-empty and disjoint. Since $\pi$ is proper, $\pi(A)$ and $\pi(B)$ are closed in $[r,\infty)$ and $[r,\infty) = \pi(A)\cup \pi(B)$. If $t \in \pi(A) \cap \pi(B)$ there are elements $a \in A \cap \pi^{-1}(t)$ and $b \in B \cap \pi^{-1}(t)$, and since $\pi^{-1}(t)$ is a simplex there is a continuous path in $\pi^{-1}(t)$ connecting $a$ and $b$. The properties of $A$ and $B$ imply that this is impossible. Hence $\pi(A)\cap \pi(B) = \emptyset$, which is also impossible since $[r,\infty)$ is connected, whence $\pi^{-1}([r,\infty))$ must be connected. For $n \in \mathbb N, \ n \geq r$, let $\overline{\pi^{-1}([n,\infty))}$ be the closure of $\pi^{-1}([n,\infty))$ in $\beta S_+$. From what we have just shown it follows that $\overline{\pi^{-1}([n,\infty))}$ is connected, and hence so is
$$
\partial S_+ = \bigcap_{n \geq r} \overline{\pi^{-1}([n,\infty))} .
$$
It follows then from Lemma \ref{05-08-23} that $\KMS_\infty(\sigma)$ is connected.
\end{proof}

It follows that $\KMS_{-\infty}(\sigma)$ is connected if $\pi(S)$ contains an interval of the form $(-\infty,r]$.

%If $\overline{\pi^{-1}([n,\infty))} \subseteq U \cup V$, where $U,V$ are open in $\beta S_+$, it follows that either $\pi^{-1}([n,\infty)) \cap U$ or $\pi^{-1}([n,\infty)) \cap V$ is empty since $\pi^{-1}([n,\infty))$ is connected. If $\pi^{-1}([n,\infty)) \cap U = \emptyset$, $\pi^{-1}([n,\infty)) \subseteq U^c$ and hence $\overline{\pi^{-1}([n,\infty))} \subseteq U^c$; i.e. $\overline{\pi^{-1}([n,\infty))} \cap U = \emptyset$.

%\bigskip

%If $\bigcap_{n \geq r} \overline{\pi^{-1}([n,\infty))} \subseteq U \cup V$, where $U,V$ are open in %$\beta S_+$, 
%$\bigcap_n \left(\overline{\pi^{-1}([n,\infty))} \backslash (U \cup V)\right) = \emptyset$, which by compactness implies that $\left(\overline{\pi^{-1}([n,\infty))} \backslash (U \cup V)\right) = \emptyset $ for some $n$, i.e. $\overline{\pi^{-1}([n,\infty))} \subseteq U \cup V$, implying that either $\overline{\pi^{-1}([n,\infty))} \cap U$ or $\overline{\pi^{-1}([n,\infty))} \cap V$ is empty. In the first case, $\bigcap_{n \geq r} \overline{\pi^{-1}([n,\infty))} \cap U = \emptyset$. 

\section{The main result}
 What we prove here is the following

\begin{thm}\label{main} Let $B$ be a simple infinite dimensional unital AF-algebra and let $(S,\pi)$ be a proper simplex bundle such that $\pi(S)$ is neither bounded above nor below and $\pi^{-1}(0)$ is affinely homeomorphic to the tracial state space of $B$. Let $D_+$ and $D_-$ be non-empty compact metric spaces such that $D_+$ is connected if $\pi(S)$ contains an interval of the form $[r,\infty)$ and $D_-$ is connected if $\pi(S)$ contains an interval of the form $(-\infty,r]$.

There is a $2\pi$-periodic flow $\sigma$ on $B$ such that the $\KMS$ bundle of $\sigma$ is isomorphic to $(S,\pi)$, the set of $\KMS_{\infty}$ states of $\sigma$ is homeomorphic to $D_+$ and the set of $\KMS_{-\infty}$ states of $\sigma$ is homeomorphic to $D_-$.
\end{thm}

 By Corollary \ref{06-08-23a} it is necessary to assume that $D_+$ is connected when $\pi(S)$ contains an interval of the form $[r,\infty)$, and the condition will be used shortly in Lemma \ref{08-08-23}, but only once in the proof of Theorem \ref{main}. When $\pi(S)$ does not contain an interval of the indicated form, $D_+$ can be any non-empty compact metric space. Similar remarks apply to $D_-$.

The proof is long and will be divided into smaller pieces. For the rest of this section, assume that we are in the setting of Theorem \ref{main}.

\subsection{Preparations}

\begin{lemma}\label{08-08-23} There are continuous surjections $t_\pm : \partial \pi(S_\pm) \to D_\pm$.
\end{lemma}
\begin{proof} To construct $t_+$ assume first that $\pi(S)$ contains an interval of the form $[r,\infty)$. In this case $D_+$ is a compact continuum by assumption and hence the image of $\partial [0,\infty)$ under a continuous surjection by a result of Aarts and Van Emde Boas, \cite{AB}. (See also Theorem 3 of \cite{DH}.) This gives the desired conclusion in this case because $\partial [0,\infty) = \partial [r,\infty) =  \partial \pi(S_+)$. Assume that $\pi(S)$ does not contain an interval of the form $[r,\infty)$.  There is then an increasing sequence $\{r_n\}$ in $[0,\infty)$ such that $\lim_{n \to \infty} r_n = \infty$, $r_n\notin \pi(S)$ and $[r_{n},r_{n+1}] \cap  \pi(S) \neq \emptyset$ for all $n$. Let $\kappa_+ : \pi(S_+) \to \mathbb N$ be a continuous function such that
$$
\kappa_+(x) = n, \ x \in [r_n,r_{n+1}] .
$$
Since $\kappa_+$ is continuous and surjective it induces a continuous surjection $\partial \pi(S_+) \to \partial \mathbb N$. Since $D_+$ is compact and separable there is a continuous surjection $\partial \mathbb N \to D_+$. The composition $\partial \pi(S_+) \to \partial \mathbb N \to D_+$ is a continuous surjection $t_+ : \partial \pi(S_+) \to D_+$. The construction of $t_-$ is identical.
\end{proof}

Note that $\pi : S \to \mathbb R$ restricts to continuous maps $\pi_\pm : S_\pm \to \pi(S_\pm)$ which induce continuous surjections $\hat{\pi}_\pm : \partial S_+ \to \partial \pi(S_\pm)$. Set 
$$
r_\pm := t_\pm \circ \hat{\pi}_\pm : \ \partial S_\pm \to D_\pm.
$$
Given a topological space $X$ we denote in the following by $C_\mathbb R(X)$ the space of continuous real-valued functions on $X$. We can then consider the maps
$r_\pm^* : C_\mathbb R(D_\pm) \to C_\mathbb R(\partial S_\pm)$ dual to $r_\pm$; i.e. $r_\pm^*(f) = f \circ r_\pm$. Given a locally compact Hausdorff space we denote by $C_{b,\mathbb R}(X)$ the set of continuous bounded real-valued functions on $X$. Let $\mathcal A_b(S,\pi)$ denote the set of bounded functions from $\mathcal A(S,\pi)$; i.e. 
$$
\mathcal A_b(S,\pi) := C_{b,\mathbb R}(S) \cap \mathcal A(S,\pi).
$$
By definition of the  Stone-\v{C}ech boundary there are canonical continuous and surjective maps $C_{b,\mathbb R}(S_\pm) \to C_\mathbb R(\partial S_\pm)$. Composed with the obvious maps $C_{b,\mathbb R}(S) \to C_{b,\mathbb R}(S_\pm)$ we get linear maps $C_{b,\mathbb R}(S) \to  C_\mathbb R(\partial S_\pm)$. By restriction this gives us linear maps  
$$
R_{\pm} : \mathcal A_b(S,\pi) \to C_\mathbb R(\partial S_{\pm}).
$$ 

\begin{lemma}\label{08-08-23a}
\begin{equation*}\label{23-06-23}
r^*_\pm(C_\mathbb R(D_\pm)) \subseteq R_\pm(\mathcal A_b(S,\pi)) 
\end{equation*}
\end{lemma}
\begin{proof} Let $q_\pm : C_{b,\mathbb R}(\pi(S_\pm)) \to C_\mathbb R(\partial \pi(S_\pm))$ be the canonical surjections. Then
\begin{equation*}
\begin{xymatrix}{
C_{b,\mathbb R}(\pi(S_\pm)) \ar[d]^{q_\pm} \ar[r]^-{\pi_\pm^*} & \mathcal A_b(S_\pm,\pi) \ar[d]^{R_\pm}\\
C_\mathbb R(\partial \pi(S_\pm)) \ar[r]^-{\hat{\pi}_\pm^*} & C_\mathbb R(\partial S_\pm) }
\end{xymatrix}
\end{equation*} 
commutes. Let $f \in C_\mathbb R(D_\pm)$. Then $r^*_\pm (f) = f \circ  t_\pm \circ \hat{\pi}_\pm$ and $f \circ  t_\pm \in C_\mathbb R(\partial \pi(S_\pm))$. Since $q_\pm$ is surjective there is an element $g \in C_{b,\mathbb R}(\pi(S_\pm))$ such that $q_\pm(g) =f \circ  t_\pm$ and then
$r^*_\pm(f) = \hat{\pi}_\pm^*(f \circ t_\pm) = R_\pm(g \circ \pi_\pm)$.
\end{proof}

\subsection{Construction of a dimension group}\label{construction}

By Lemma \ref{08-08-23a} we can choose linear maps 
$$
l_\pm : r^*_\pm(C_\mathbb R (D_\pm)) \to \mathcal A_b(S,\pi)
$$
such that $l_\pm(1) = 1$ and $l_\pm$ is a right-inverse for $R_\pm$, i.e.
$$
R_\pm \circ l_\pm(g) = g \ \ \ \forall g\in r^*_\pm(C_\mathbb R (D_\pm))  .
$$
For each $k \in \mathbb N$, choose continuous functions $\psi^0_k, \psi^\pm_k : \mathbb R \to [0,1]$ such that

\begin{itemize}
\item[(A)] $\psi^0_k(t) = 1, \ -\frac{1}{2k} \leq t \leq \frac{1}{2k}$, 
\item[(B)] $\psi^-_k(t) = 1, \ t \leq - \frac{1}{k}$,
\item[(C)] $\psi^+_k(t) = 1, \ t \geq \frac{1}{k}$, and
\item[(D)] $\psi_k^-(t) + \psi_k^0(t) + \psi_k^+(t) = 1$ for all $t \in \mathbb R$. 
\end{itemize}
For each $n \in \mathbb Z$ we fix also a continuous function $\varphi_n : \mathbb R \to [0,1]$ such that $\varphi_n(x) = 1$ for $x \leq n$ and $\varphi_n(x) = 0$ for $x \geq n+1$.

%By exchanging $l_\pm$ by $\psi^\pm_1\circ \pi l_\pm$ we may assume that $\supp l_+(f) \subseteq \pi^{-1}([-1,\infty))$ for all $f \in r^+(C_\mathbb R(D_+))$ and $\supp l_-(f) \subseteq \pi^{-1}((-\infty,1])$ for all $f \in r^-(C_\mathbb R(D_-))$. In particular,
%$$
%R_\mp \circ l_\pm = 0.
%$$

In what follows we denote the support of a real-valued function $f$ by $\supp f$, and we denote by $\mathcal A_c(S,\pi)$ the set of elements $f$ of $\mathcal A(S,\pi)$ with $\supp f$ compact.

\begin{lemma}\label{helsingor} There are countable additive subgroups $C_\pm \subseteq r^*_\pm(C_\mathbb R(D_\pm))$ and $\mathcal C \subseteq \mathcal A_c(S,\pi)$ such that
\begin{itemize}
\item[(a)] $1 \in C_\pm$.
\item[(b)] $C_\pm$ is dense in $r^*_\pm(C_\mathbb R(D_\pm))$.
\item[(c)] $q C_\pm \subseteq C_\pm$ and $q\mathcal C \subseteq \mathcal C$ for all rational numbers $q\in \mathbb Q$.
\item[(d)] For every $a \in  \mathcal A_c(S,\pi)$ and $\epsilon >0$ there is an $a' \in \mathcal C$ such that 
$$
\sup_{s \in S} \left|a(s)-a'(s)\right| \leq \epsilon.
$$
\item[(e)] $\varphi_n \circ \pi\mathcal C \subseteq \mathcal C $ for all $n \in \mathbb Z$.
%\item[(f)] $\psi^0_1 \circ \pi \in \mathcal C$.
\item[(f)] $\psi^\pm_k \circ \pi - \psi^\pm_1 \circ \pi \in \mathcal C$ for all $k \in \mathbb N$.  
\end{itemize}
\end{lemma}
\begin{proof} This follows from the separability of $C_\mathbb R(D_\pm)$ and the second countability of the topology of $S$. We leave the details to the reader.
\end{proof}

Set
$$
E_\pm := \psi^\pm_1\circ \pi l_\pm(C_\pm) .
$$
and
$$
A(\pm) :=  \mathcal C + E_\pm.
$$

\begin{lemma}\label{helsingor1} Let $a \in A(\pm)$. Then $a \in \mathcal C$ if and only if $R_\pm(a) = 0$.
\end{lemma}
\begin{proof} Write $a =  c + \psi_1^{\pm}\circ \pi l_\pm(b_\pm)$, where $c \in \mathcal C$ and $b_\pm \in C_\pm$. Then $R_\pm(a) = b_\pm$. Thus, if $R_\pm(a) = 0$, $a = c \in \mathcal C$. The converse is trivial.
\end{proof}

%\begin{lemma}\label{helsingor1} Let $a \in A$. Then $a \in \mathcal C$ if and only if $R_\pm(a) = 0$.
%\end{lemma}
%\begin{proof} Write $a = l_-(b_-)- + c + l_+(b_+)$, where $c \in \mathbb C$ and $b_\pm \in C_\pm$. Then $R_+(a) = b_+$ and $R_-(b_-) = b_-$. Thus, if $R_\pm(a) = a$, $a = c \in \mathcal C$. The converse is trivial.
%\end{proof}

Let $G(\pm)$ be the set of Laurent polynomials in $e^\pi$ with $A(\pm)$-coefficients. That is, $G(\pm)$ consists of the elements $p \in \mathcal A(S,\pi)$ of the form
\begin{equation}\label{holbaek3}
p(s) = \sum_{n \in \mathbb Z} p_n(s) e^{n\pi(s)} ,
\end{equation}
where $p_n \in A(\pm)$ for all $n$ and $p_n =0 $ for all but finitely many $n$. Note that we are here defining two different subgroups of $\mathcal A(S,\pi)$, namely $G(+)$ which is the set of Laurent polynomials in $e^\pi$ with $A(+)$-coefficients and $G(-)$ which is the set of Laurent polynomials in $e^\pi$ with $A(-)$-coefficients. As we have done already, we use the notation $\pm$ in the same spirit throughout. Note that the support of an element of $G(+)$ is bounded below while the support of an element of $G(-)$ is bounded above.

\begin{lemma}\label{roervigx} Let $p \in G(\pm)$. Assume that $p \notin \mathcal A_c(S,\pi)$. There is a unique integer $N_\pm(p) \in \mathbb Z$ such that $e^{-N_\pm(p)\pi}p \in \mathcal A_b(S,\pi)$ and 
$$
R_\pm(e^{-N_\pm(p)\pi}p) \neq 0.
$$ 

\smallskip

Write $p = \sum_{n \in \mathbb Z} p_ne^{n\pi}$, where $p_n \in A(\pm)$ for all $n$ and $p_n =0 $ for all but finitely many $n$. Then
\begin{equation}\label{18-07-23}
R_\pm(e^{-N_\pm(p)\pi}p) = R_\pm(p_{N_\pm(p)}).
\end{equation}
Furthermore, in the plus case
 \begin{align*}
& N_+(p)= \max \{n \in \mathbb Z: \ R_+(p_n) \neq 0 \}  = \max \{n \in \mathbb Z: \ p_n \notin \mathcal C\},
\end{align*}
and in the minus case
\begin{align*}
& N_-(p)= \min\{n \in \mathbb Z: \ R_-(p_n) \neq 0 \} = \min \{n \in \mathbb Z: \ p_n \notin \mathcal C\}.
\end{align*}

\end{lemma}
 \begin{proof} We give the details in the plus case. It follows from Lemma \ref{helsingor1} that $\{n \in \mathbb Z: \ R_+(p_n) \neq 0 \}$ is not empty since $p \notin \mathcal A_c(S,\pi)$. Set 
$$
N_+(p) := \max \{n \in \mathbb Z: \ R_+(p_n) \neq 0 \}.
$$ 
It follows from Lemma \ref{helsingor1} that $N_+(p) =\max \{n \in \mathbb Z: \ p_n \notin \mathcal C\}$. 
Then $e^{-N_+(p)\pi}p  =  h + p_{N_+(p)} + \sum_{j=1}^\infty p_{N_+(p)-j}e^{-j \pi}$ for some $h \in \mathcal A_c(S,\pi)$, showing first that $e^{-N_+(p)\pi}p \in \mathcal A_b(S,\pi)$ and then that 
$$
R_+(e^{-N_+(p)\pi}p) = R_+(p_{N_+(p)}) \neq 0.
$$ 

To prove uniqueness let $k \in \mathbb Z$ have the properties that $e^{-k\pi}p \in \mathcal A_b(S,\pi)$ and $R_+(e^{-k\pi}p) \neq 0$. If $n> N_+(p)$, the identity 
$$
e^{-n\pi}p = e^{(-n+N_+(p))\pi}e^{-N_+(p)\pi}p 
$$ 
shows that $R_+(e^{-n\pi}p)= 0$. Hence $k \leq N_+(p)$. Assume for a contradiction that $k < N_+(p)$. Since $R_+(p_{N_+(p)}) \neq 0$ there is a $\delta > 0$ and a sequence $\{s_n\}$ in $S$ such that $\lim_{n \to \infty} \pi(s_n) = \infty$ and $|p_{N_+(p)}(s_n)| \geq \delta$ for all $n$. Write
$$
e^{-k\pi}p = h + p_{N_+(p)}e^{(N_+(p)-k)\pi} + \sum_{j=1}^\infty p_{N_+(p)-j} e^{(N_+(p)-j-k)\pi} 
$$
where $h \in \mathcal A_c(S,\pi)$.
Since $\lim_{n \to \infty} \pi(s_n) = \infty$ we see that
$$
\left|h(s) + \sum_{j=1}^\infty p_{N_+(p)-j}(s_n) e^{(N_+(p)-j-k)\pi(s_n)}\right| < e^{(N_+(p)-k)\pi(s_n)}\frac{\delta}{2}
$$
for all large $n$. Therefore 
$$
|e^{-k\pi(s_n)}p(s_n)| \geq (|p_{N_+(p)}(s_n)| - \frac {\delta}{2})e^{(N_+(p)-k)\pi(s_n)} \geq \frac{\delta}{2} e^{(N_+(p)-k)\pi(s_n)} ,
$$
for all large $n$, contradicting that $e^{-k\pi}p$ is bounded.
 \end{proof}

  For $p \in G(\pm) \backslash \mathcal A_c(S,\pi)$, given by a sum as in \eqref{holbaek3}, set 
$$
\mathbb L_+(p) := \max \{n \in \mathbb Z: \ R_+(p_n) \neq 0 \}
$$
in the plus case and
$$
\mathbb L_-(p) := \min \{n \in \mathbb Z: \ R_-(p_n) \neq 0 \}
$$
in the minus case. For $p \in G(\pm) \cap \mathcal A_c(S,\pi)$, set
$$
\mathbb L_\pm(p) = 0.
$$
It follows from Lemma \ref{roervigx} that $\mathbb L_\pm(p)$ does not depend on how $p$ is realized as a Laurent polynomial as in \eqref{holbaek3}. Note also that $\mathbb L_+(p)$ only depends on the behaviour of $p$ close to $+\infty$: If there is an $R > 0$ such that $p(s) = q(s)$ when $\pi(s) \geq R$, then $\mathbb L_+(p) = \mathbb L_+(q)$. A similar remark applies to $\mathbb L_-$.

 Let $T \in \mathbb R$, $T > 0$. For $p \in G(+)$ we write $0 \prec p$ when there is an $\epsilon > 0$ such that
\begin{equation}\label{11-01-23+}
e^{-\mathbb L_+(p)\pi(s)}p(s) \geq \epsilon \ \ \text{for all} \ s  \in \pi^{-1}([T,\infty)).
\end{equation}
Similarly, for $p \in G(-)$ we write $0 \prec p$ when there is an $\epsilon > 0$ such that
\begin{equation}\label{11-01-23-}
e^{-\mathbb L_-(p)\pi(s)}p(s) \geq \epsilon \ \text{for all} \ s \in \pi^{-1}((-\infty,-T]).
\end{equation}

Set
$$
G(\pm)^+_T := \left\{p \in G(\pm) : \ 0 \prec p\right\} \cup \{0\} .
$$

\begin{lemma}\label{21-07-23} Let $p \in G(\pm)$ and write $p$ as in \eqref{holbaek3}. The following conditions (a), (b) and (c) are equivalent.

\begin{itemize}
\item[(a)] $p \in G(\pm)^+_T \backslash \{0\}$.
\item[(b)] 
\begin{itemize}
\item[(1)]  $p(s) > 0$ when $\pm \pi(s) \geq T$, and
\item[(2)]  $R_\pm(p_{\mathbb L_\pm(p)})$ is strictly positive everywhere on $\partial S_\pm$.
\end{itemize}
\item[(c)]
\begin{itemize}
\item[(1)]  $p(s) > 0$ when $\pm \pi(s) \geq T$, and
\item[(2)]  $R_\pm(e^{-\mathbb L_\pm(p)\pi}p)$ is strictly positive everywhere on $\partial S_\pm$.
\end{itemize}
\end{itemize} 
\end{lemma}
\begin{proof} (b) $\Leftrightarrow$ (c) follows because 
\begin{equation*}\label{11-01-23bx}
\lim_{\pm \pi(s) \to \infty} \left(p_{\mathbb L_\pm(p)}(s)  - e^{-\mathbb L_\pm(p)\pi(s)}p(s)\right) = 0.
\end{equation*}
(a) $\Rightarrow$ (c) is obvious and (c) $\Rightarrow$ (a) is easy.
\end{proof}

\begin{lemma}\label{09-01-23x} The pair $(G(\pm),G(\pm)^+_T)$ has the following properties:
\begin{itemize}
\item[(0)]  $G(\pm)^+_T + G(\pm)^+_T = G(\pm)^+_T$,
\item[(1)] $G(\pm)^+_T \cap (-G(\pm)^+_T) = \{0\}$,
\item[(2)] $G(\pm)^+_T - G(\pm)^+_T = G(\pm)$,
\item[(3)] $f\in G(\pm), \ n \in \mathbb N, \ nf \in G(\pm)^+_T \Rightarrow f \in G(\pm)^+_T$.
\end{itemize}
\end{lemma}
\begin{proof} We consider only the plus case. (1) and (3) are trivial. (0): Let $p,q \in G(+)^+_T \backslash \{0\}$. Assume $\mathbb L_+(p) \geq \mathbb L_+(q)$. Then 
$$
e^{-\mathbb L_+(p)\pi}(p+q) =  e^{-\mathbb L_+(p)\pi}p  + e^{(\mathbb L_+(q) - \mathbb L_+(p))\pi} e^{-\mathbb L_+(q)\pi}q \in \mathcal A_b(S,\pi)
$$
and
$$
R_+( e^{-\mathbb L_+(p)\pi}(p+q)) \geq R_+( e^{-\mathbb L_+(p)\pi}p) > 0 ,
$$
showing that $\mathbb L_+(p+q) = \mathbb L_+(p)$. Note that $e^{-\mathbb L_+(p+q)\pi(s)}(p(s)+q(s)) > 0$ for all $s \in \pi^{-1}([T,\infty))$ and 
$$
R_+(e^{-\mathbb L_+(p+q)\pi}(p+g)) = R_+(e^{-\mathbb L_+(p)\pi}(p+q))
$$
 is strictly positive. It follows that $p+q \in G(+)^+_T \backslash \{0\}$ by Lemma \ref{21-07-23}. (2): Let $p \in G(+)\backslash \{0\}$ and write $p$ as in \eqref{holbaek3}. Let $M$ be the largest $n$ for which $p_n \neq 0$. Set 
 $$
 q(s):= h(s)  + \psi^+_1(\pi(s)) e^{(M+1)\pi(s)}
 $$
 where $h\in \mathcal C$ is chosen such that $q(s) > |p(s)|$ when $\pi(s) \geq 0$. This is possible by definition of $M$ and condition (d) of Lemma \ref{helsingor}. It follows from (a) of Lemma \ref{helsingor} and the definition of $E_+$ that $q \in G(+)$. Since $\mathbb L_+(q) = \mathbb L_+(q-p) = M+1$ and
$$
R_+(e^{-(M +1)\pi}q) = R_+(e^{-(M +1)\pi}(q-p)) = 1,
$$
it follows from Lemma \ref{21-07-23} that $q, q-p \in G(+)^+_T$. Since $p = q- (q-p)$ this proves (2).
\end{proof}

We need some preparations to show that $(G(\pm),G(\pm)^+_T)$ are dimension groups. For this we shall use results as well as terminology from \cite{EHS}; in particular the notion of strong Riesz groups. 
Given a ordered group $(H,H^+)$ we can consider the direct sum $\oplus_{\mathbb Z} H$ with the lexiographic order $<_{lex}$. The corresponding semi-group of positive elements is
$$
\left(\oplus_{\mathbb Z} H\right)^+_{lex} = \{0\} \cup \left\{ f \in \oplus_{\mathbb Z} H : \ f_n \in H^+ \ \text{when} \ n = \max \{k: \ f_k \neq 0\} \right\}.
$$

\begin{lemma}\label{11-01-23f} If $(H,H^+)$ is a strong Riesz group then so is $\left( \oplus_{\mathbb Z} H, \left(\oplus_{\mathbb Z} H\right)^+_{lex}\right)$.
\end{lemma} 
\begin{proof} It follows from Theorem 3.10 of \cite{E3} that $\left( \oplus_{\mathbb Z} H, \left(\oplus_{\mathbb Z} H\right)^+_{lex}\right)$ is an ordered group with the Riesz decomposition property. (A different and slightly more elaborate argument, which establishes the Riesz interpolation property instead, is given in the proof of Lemma 3.2 of \cite{BEK1}.) Since $\left( \oplus_{\mathbb Z} H, \left(\oplus_{\mathbb Z} H\right)^+_{lex}\right)$ is unperforated because $(H,H^+)$ is, it follows that $\left( \oplus_{\mathbb Z} H, \left(\oplus_{\mathbb Z} H\right)^+_{lex}\right)$ is a Riesz group. By Corollary 1.2 of \cite{EHS} it is therefore sufficient here to show that $\left( \oplus_{\mathbb Z} H , \left(\oplus_{\mathbb Z} H\right)^+_{lex}\right)$ is prime and without minimal elements. It follows from Corollary 1.2 of \cite{EHS} that $(H,H^+)$ has these properties. Let $h =(h_n) \in (\oplus_\mathbb Z H)^+_{lex} \backslash \{0\}$ and set $n := \max\{k: h_k \neq 0\}$. Since $(H,H^+)$ does not have minimal elements there is an element $h'$ such that $0 <h'<h_n$ in $(H,H^+)$. Define $h'' \in \oplus_\mathbb Z H$ such that $h''_j = 0, \ j \neq n$, and $h''_n = h'$. Then $0 < h'' < h$ in $\left( \oplus_{\mathbb Z} H, \left(\oplus_{\mathbb Z} H\right)^+_{lex}\right)$, showing that $\left( \oplus_{\mathbb Z} H, \left(\oplus_{\mathbb Z} H\right)^+_{lex}\right)$ has no minimal elements. To show that $\left( \oplus_{\mathbb Z} H, \left(\oplus_{\mathbb Z} H\right)^+_{lex}\right)$ is prime with the aid of Proposition 1.1 in \cite{EHS}, consider two elements $a =(a_n), \ b = (b_n)$ in $ \left(\oplus_{\mathbb Z} H\right)^+_{lex} \backslash \{0\}$. If $\max \{k : \ a_k \neq 0\} = \max \{k : \ b_k \neq 0\} =: n$ we use the strong Riesz interpolation property of $(H,H^+)$ to find an element $h' \in H^+$ such that $0 < h' < a_n$ and $0< h'<b_n$. We define $h''$ as above and note that $0 < h'' \leq a, b$ in $\left( \oplus_{\mathbb Z} H, \left(\oplus_{\mathbb Z} H\right)^+_{lex}\right)$. If $\max \{k : \ a_k \neq 0\} >  \max \{k : \ b_k \neq 0\} =: n$ we use that $H^+$ has no minimal elements to find $h' \in H^+$ such that $0 < h' < b_n $ in $(H,H^+)$. Defining $h''$ as before we find that  $0 < h'' \leq a, b$ in $\left( \oplus_{\mathbb Z} H, \left(\oplus_{\mathbb Z} H\right)^+_{lex}\right)$. It follows now from Proposition 1.1 in \cite{EHS} that $\left( \oplus_{\mathbb Z} H, \left(\oplus_{\mathbb Z} H\right)^+_{lex}\right)$ is prime.
\end{proof}

Since $C_\pm$ is dense in $r^*_\pm(C_\mathbb R(D_\pm))$, it is wellknown and easy to see that  $(C_\pm, C_\pm^+)$ is a strong Riesz group when we set
$$
C_\pm^+ := \left\{ h \in C_\pm : \ h \geq \epsilon \ \text{for some} \ \epsilon > 0 \right\} \cup \{0\},
$$
cf. \cite{EHS}. It follows therefore from Lemma \ref{11-01-23f} that $(\oplus_\mathbb Z C_\pm , (\oplus_\mathbb Z C_\pm)^+_{lex})$ is a strong Riesz group. It is this version of Lemma \ref{11-01-23f} we use below.

\begin{lemma}\label{osterbro2} Set 
$$
(\oplus_\mathbb Z A(+))^{++} := \left\{ (p_n) \in \oplus_\mathbb Z A(+): \ (R_+(p_n)) \in (\oplus_\mathbb Z C_+)^+_{lex}\backslash \{0\}\right\} \cup \{0\}.
$$ 
Then $\left(\oplus_\mathbb Z A(+),(\oplus_\mathbb Z A(+))^{++}\right)$ is a strong Riesz group. Similarly, set
$$
(\oplus_\mathbb Z A(-))^{++} := \left\{ (p_n) \in \oplus_\mathbb Z A(-): \ (R_-(p_{-n})) \in (\oplus_\mathbb Z C_-)^+_{lex} \backslash \{0\}\right\} \cup \{0\}.
$$ 
Then $\left(\oplus_\mathbb Z A(-),(\oplus_\mathbb Z A(-))^{++}\right)$ is a strong Riesz group.
\end{lemma}
\begin{proof} This follows now from Lemma \ref{11-01-23f} above and Lemma 3.2 in \cite{EHS}. (Note the minus sign of $p_{-n}$ in the definition of $(\oplus_\mathbb Z A(-))^{++}$ which is introduced in order to use the lexiographic order rather than the colexiographic order in the minus case.)
\end{proof}

Lemma \ref{osterbro2} will be used to prove the following

\begin{lemma}\label{skovshoved5} $(G(\pm),G(\pm)^+_T)$ are strong Riesz groups.
\end{lemma}
\begin{proof} Consider the plus case. In view of Lemma \ref{09-01-23x} it suffices to prove
\begin{itemize}
\item[(4)] $(G(+),G(+)_T^+)$ has the strong Riesz interpolation property.
\end{itemize}
Consider therefore elements $f^i,g^j \in G(+), \ i,j \in \{1,2\}$, such that $f^i < g^j$ in $(G(+),G(+)^+_T)$ for all $i,j$. Then $0 \prec g^j-f^i$ for all $i,j$, and it follows from Lemma \ref{roervigx} that
$$
\mathbb L_+(g^j -f^i) = \max\{ n\in \mathbb Z: R_+(f^i_n) \neq R_+(g^j_n) \}
$$
and from (b)(2) of Lemma \ref{21-07-23} that
$$
R_+(g^j_{\mathbb L_+(g^j -f^i)}) - R_+(f^i_{\mathbb L_+(g^j -f^i)}) \in C_+^+ \backslash \{0\} .
$$
Thus $\left(f^i_n\right)_{n \in \mathbb N}  < \left(g^j_n\right)_{n \in \mathbb N}$ in $\left(\oplus_\mathbb Z A(+),(\oplus_\mathbb Z A(+))^{++}\right)$ for all $i,j$.
It follows in this way from Lemma \ref{osterbro2} that there is an element $h \in G(+)$ such that 
$$
\left(f^i_n\right)_{n \in \mathbb N}  <  \left(h_n\right)_{n \in \mathbb N} < \left(g^j_n\right)_{n \in \mathbb N}
$$
in $\left(\oplus_\mathbb Z A(+),(\oplus_\mathbb Z A(+))^{++}\right)$ for all $i,j$. Then 
$$
R_+\left(e^{-\mathbb L_+(h-f^i)\pi}\left( h- f^i\right)\right) =R_+(h_{\mathbb L_+(h -f^i)}) - R_+(f^i_{\mathbb L_+(h -f^i)})
$$
are strictly positive on $\partial S_+$ for both $i$'s and there is therefore an $N_i\in \mathbb N$ and an $\epsilon_i > 0$ such that
$$
 e^{-\mathbb L_+(h-f^i) \pi(s)}\left( h(s) - f^i(s)\right) \geq \epsilon_i
 $$
 when $\pi(s) \geq N_i$, $i =1,2$. Handling the $g^j$'s in the same way we find an $N \in \mathbb N$ and an $\epsilon > 0$ such that
\begin{equation}\label{dragoer}
  e^{-\mathbb L_+(h-f^i) \pi(s)}\left( h(s) - f^i(s)\right) \geq \epsilon
\end{equation}
  and
\begin{equation}\label{dragoer1}
 e^{-\mathbb L_+(g^j -h) \pi(s)}\left( g^j(s) - h(s)\right) \geq \epsilon
\end{equation}
 for all $i,j$ when $\pi(s) \geq N$. We can safely assume that $N > T$. 
 
 It follows from Lemma \ref{21-07-23} that $f^i(s) < g^j(s)$ for all $s \in \pi^{-1}([T,\infty))$  and all $i,j$. Hence $f^i(s) - h(s) < g^j(s) - h(s)$ for all $i,j$ when $T \leq \pi(s) \leq N+1$. It follows therefore from Lemma 4.4 in \cite{ET} that there is a $w \in \mathcal A_c(S,\pi)$ such that
$$
f^i(s) - h(s) < w(s) <  g^j(s) - h(s)
$$
for all $i,j$ and $s \in \pi^{-1}([T,N+1])$. By using (d) of Lemma \ref{helsingor} we can arrange that $w \in \mathcal C$. Set 
$$
h' : = \varphi_N\circ \pi (h +w) + (1-\varphi_N \circ \pi) h
$$
Then $h' \in G(+)$ since $h'-h= \varphi_N \circ \pi w \in \mathcal C$ by (e) of Lemma \ref{helsingor}, and 
\begin{equation}\label{maarup}
 f^i(s) < h'(s) < g^j(s)
\end{equation}
 for all $i,j$ and all $s \in \pi^{-1}([T,\infty))$. Note that $\mathbb L_+(h'-f^i) = \mathbb L_+(h-f^i)$ and $\mathbb L_+(g^j-h') = \mathbb L_+(g^j-h)$. It follows therefore from \eqref{dragoer}, \eqref{dragoer1} and \eqref{maarup} that
$$
  e^{-\mathbb L_+(h'-f^i) \pi(s)}\left( h'(s) - f^i(s)\right) \geq \epsilon'
  $$
  and
$$
 e^{-\mathbb L_+(g^j -h') \pi(s)}\left( g^j(s) - h'(s)\right) \geq \epsilon'
 $$
 for all $i,j$ and all $s \in \pi^{-1}([T,\infty))$ provided $\epsilon' >0$ is small enough. Thus $f^i < h' < g^j$ in $(G(+),G(+)^+_T)$ for all $i,j$.
 
 The minus case is analogous and we leave it to the reader.
\end{proof}

\begin{lemma}\label{holbaekx} $(1-e^{-\pi}) G(+) \subseteq G(+)$ and $(1-e^{-\pi}) G(+)_T^+ \subseteq G(+)_T^+$.
\end{lemma}
\begin{proof} The first inclusion follows from the definition of $G(+)$. For the second, let $p \in G(+)_T^+ \backslash \{0\}$. Since $\supp ((1 - e^{-\pi})p) = \supp p$, it follows that $(1-e^{-\pi})p \notin  \mathcal A_c(S,\pi)$. Write $p$ as in \eqref{holbaek3}. Then 
$$
(1-e^{-\pi})p = \sum_{n \in \mathbb Z} (p_n -p_{n+1})e^{n\pi}
$$ 
and hence
\begin{align*}
&\mathbb L_+((1-e^{-\pi})p) = \max \left\{ n \in \mathbb Z: \ R_+(p_n - p_{n+1}) \neq 0 \right\} \\
&= \max \left\{ n \in \mathbb Z: \ R_+(p_n) \neq 0 \right\} = \mathbb L_+(p) .
\end{align*}
Since $R_+(e^{-\mathbb L_+(p)\pi}(1-e^{-\pi})p) = R_+(e^{-\mathbb L_+(p)\pi}p) $ is strictly positive on $\partial S_+$ and $(1-e^{-\pi(s)}) p(s) > 0$ for all $s \in \pi^{-1}([T,\infty))$, it follows from Lemma \ref{21-07-23} that $(1-e^{-\pi})p \in G(+)_T^+$.
\end{proof}

A similar proof gives
\begin{lemma}\label{skovshoved6} $(1-e^{\pi}) G(-) \subseteq G(-)$ and $(1-e^{\pi}) G(-)_T^+ \subseteq G(-)_T^+$.
\end{lemma}

From the proofs of the last lemmas we take the observation that
\begin{equation}\label{holbaek1xx}
\mathbb L_\pm((1-e^{\mp \pi})p) = \mathbb L_\pm(p)
\end{equation}
for all $p \in G(\pm)$. Define $\varphi_\pm : G(\pm) \to G(\pm)$ by $\varphi_\pm(p) = (1-e^{\mp \pi})p$. Thanks to Lemma \ref{skovshoved6} and Lemma \ref{holbaekx} we can consider the inductive limit(s)  
$$
(\mathcal G(\pm),\mathcal G(\pm)_T^+)
$$ 
in the category of ordered groups of the sequence(s)
\begin{equation*}
\begin{xymatrix}{
(G(\pm),G(\pm)_T^+) \ar[r]^-{\varphi_\pm} &   (G(\pm),G(\pm)_T^+) \ar[r]^-{\varphi_\pm} &  (G(\pm),G(\pm)_T^+) \ar[r]^-{\varphi_\pm} &  \cdots .}
\end{xymatrix}
\end{equation*}
We note that $\mathcal G(\pm)$ can be identified with the additive group 
$$
A(\pm)[e^\pi, (1-e^{-\pi})^{-1}]
$$ 
of Laurent polynomials in $e^\pi$ and $(1-e^{-\pi})^{-1}$ with coefficients from $A(\pm)$. That is, the elements of $\mathcal G(\pm)$ are functions on $S \backslash \pi^{-1}(0)$ of the form
$$
\sum_{k,l \in \mathbb Z} a_{kl} e^{k \pi}(1-e^{-\pi})^l ,
$$
where the $a_{kl} \in A(\pm)$ are non-zero only for finitely many $k,l$. For every element $f \in \mathcal G(\pm)$, the products
$$
(1-e^{\mp \pi})^N f
$$
are in $G(\pm)$ for all large $N \in \mathbb N$. Thanks to \eqref{holbaek1xx} we can define $\mathbb L_\pm(f)$ for $f \in \mathcal G(\pm)$ by
$$
\mathbb L_\pm(f) := \mathbb L_\pm((1-e^{\mp \pi})^N f) 
$$
for any $N\in \mathbb N$ such that $(1-e^{\mp \pi})^N f \in G(\pm)$. The function $e^{-\mathbb L_\pm(f)\pi(s)}f(s)$ is then bounded for $\pm \pi(s) \geq 1$ but it may be unbounded near $\pi^{-1}(0)$. In order to define $R_\pm$ on functions of this type, set
$$
R_\pm(h):= R_\pm(\psi^\pm_1\circ \pi h)
$$
when $\psi^\pm_1\circ \pi h \in C_{b,\mathbb R}(S)$. This is consistent with the meaning of the symbol $R_\pm$ as we have used it so far. Then $\mathcal G(\pm)_T^+$ consists of $0$ and the elements $g \in \mathcal G(\pm)$ with the properties that
\begin{itemize}
\item[(1)] $R_\pm(e^{-\mathbb L_\pm(g)\pi}g)$ is strictly positive on $\partial S_\pm$, and
\item[(2)] $g(s) > 0$ when $\pm \pi(s) \geq T$.
\end{itemize}

\begin{lemma}\label{skovshoved9} $(\mathcal G(\pm),\mathcal G(\pm)_T^+)$ are strong Riesz groups. 
\end{lemma}
\begin{proof}
Since inductive limits preserve the properties defining a dimension group it follows from Lemma \ref{skovshoved5} that $(\mathcal G(\pm),\mathcal G(\pm)_T^+)$ is a dimension group. To establish the strong Riesz interpolation property, consider elements $f_i,g_j, \ i,j \in \{1,2\}$, such that $f_i < g_j$ in $(\mathcal G(\pm),\mathcal G(\pm)_T^+)$ for all $i,j$. There then an $N \in \mathbb N$ such that $(1-e^{\mp \pi})^Nf_i < (1-e^{\mp \pi})^Ng_j$ in $(G(\pm),G(\pm)_T^+)$ for all $i,j$. Since $(G(\pm),G(\pm)_T^+)$ is a strong Riesz group by Lemma \ref{skovshoved5} there is an element $h \in G(\pm)$ such that 
 $(1-e^{\mp \pi})^Nf_i < h < (1-e^{\mp \pi})^Ng_j$ in $(G(\pm),G(\pm)_T^+)$ for all $i,j$. Set $h' := (1-e^{\mp \pi})^{-N}h$. Then  $f_i < h' < g_j$ in $(\mathcal G(\pm),\mathcal G(\pm)_T^+)$ for all $i,j$.
\end{proof}

\bigskip

Given a Choquet simplex $\Delta$, as usual denote by $\Aff \Delta$ the set of real-valued continuous and affine functions on $\Delta$. Fix a proper simplex bundle $(S,\pi)$ with $\pi^{-1}(0)$ non-empty. Let $H$ be a countable torsion free abelian group and $\theta : H \to \Aff \pi^{-1}(0)$ a homomorphism. We assume
\begin{itemize}
\item[(1)] there is an element $u \in H$ such that $\theta(u) = 1$, and
\item[(2)] $\theta(H)$ is dense in $\Aff \pi^{-1}(0)$.
\end{itemize}
Set 
$$
H^+ := \left\{h \in H: \ \theta(h)(x) > 0 \ \ \forall x \in \pi^{-1}(0) \right\} \cup \{0\} \ .
$$
Then $(H,H^+)$ is a simple dimension group; cf. \cite{EHS}.

It follows from Lemma 4.4 in \cite{ET} that the map $r : \mathcal A_c(S,\pi) \to \Aff \pi^{-1}(0)$ given by restriction is surjective and we can therefore choose a linear map $L :  \Aff \pi^{-1}(0) \to \mathcal A_c(S,\pi)$ such that $r\circ L = \id$. We arrange, as we can, that $L(1)(s) = 1$ for all $s \in S$ with $|\pi(s)| \leq 1$. Define $\hat{L} : \bigoplus_{\mathbb Z} H \to \mathcal A(S,\pi)$ by
$$
\hat{L}\left((h_n)_{n \in \mathbb Z}\right)(s) := \sum_{n \in \mathbb Z} L(\theta(h_n))(s)e^{n\pi(s)} \ .
$$ 
 %For each $k \in \mathbb N$, choose continuous functions $\psi^0_k, \psi^\pm_k : \mathbb R \to [0,1]$ such that 
%\begin{itemize}
%\item[{}] $\psi^0_k(t) = 1, \ -\frac{1}{2k} \leq t \leq \frac{1}{2k}$, 
%\item[{}] $\psi^-_k(t) = 1, \ t \leq - \frac{1}{k}$,
%\item[{}] $\psi^+_k(t) = 1, \ t \geq \frac{1}{k}$, and
%\item[{}] $\psi_k^-(t) + \psi_k^0(t) + \psi_k^+(t) = 1$ for all $t \in \mathbb R$. 
%\end{itemize}
For $k \in \mathbb N$ consider the countable subgroup of $ \mathcal A(S,\pi)$
$$
G_k :=  \mathcal G(-)\psi^-_k \circ \pi + \hat{L}(\bigoplus_{\mathbb Z} H ) \psi^0_k \circ \pi  + \mathcal G(+) \psi^+_k \circ \pi \ ,
$$
where the $\psi_k^0,\psi^\pm_k$ are the functions from (A)-(D) above.
 Let $\mathcal A_{00}(S,\pi)$ denote the set of elements $f$ from $\mathcal A(S,\pi)$ for which $\supp f$ is compact and contained in $S \backslash \pi^{-1}(0)$. Since the topology of $S$ is second countable we can choose a countable subgroup $G_{00}$ of $\mathcal A_{00}(S,\pi)$ with the following density property:

\begin{TOM}\label{07-09-21} For all $N \in \mathbb N$, all $\epsilon > 0$ and all $f \in \mathcal A_{00}(S,\pi)$ with 
$\supp f \subseteq \pi^{-1}(]-N,N[ \backslash \{0\})$, there is $g \in G_{00}$ such that 
$$\sup_{s \in S} |f(s) -g(s)| < \epsilon
$$ 
and $\supp g \subseteq \pi^{-1}(]-N,N[ \backslash \{0\})$.
\end{TOM}

When $f_1 \in \hat{L}(\bigoplus_{\mathbb Z} H ) $ and $f^\pm_2 \in  \mathcal G(\pm)$, the difference between 
$$
f^-_2 \psi^-_k \circ \pi + f_1\psi^0_k \circ \pi  + f^+_2\psi^+_k \circ \pi
$$ 
and 
$$
f^-_2 \psi^-_{k+1} \circ \pi + f_1\psi^0_{k+1} \circ \pi  + f_2^+\psi^+_{k+1} \circ \pi
$$ 
is an element of $\mathcal A_{00}(S,\pi)$, and so by enlarging $G_{00}$ we can ensure that
\begin{equation*}\label{01-09-21c}
G_k + G_{00} \subseteq G_{k+1} + G_{00} \ .
\end{equation*}
Let $\rho_0 \in \Aut \mathcal A(S,\pi)$ be defined by
$$
\rho_0(f)(s) := e^{-\pi(s)}f(s) \ .
$$
Since $\rho_0(\mathcal A_{00}) =\mathcal A_{00}$ and since 
$(1-e^{-\pi})^{-1}\mathcal A_{00}(S,\pi) \subseteq \mathcal A_{00}(S,\pi)$,
we can enlarge $G_{00}$ further to achieve that 
\begin{equation}\label{07-09-21d}
\rho_0(G_{00}) = G_{00}
\end{equation}
and that 
\begin{equation}\label{07-09-21e}
(\id - \rho_0)(G_{00}) = G_{00} .
\end{equation}
We define
$$
G := \bigcup_{k=1}^\infty (G_k + G_{00}) .
$$
Let $\sigma \in \Aut \left( \bigoplus_\mathbb Z H\right)$ denote the shift:
$$
\sigma\left((h_n)_{n \in \mathbb Z}\right) = \left( h_{n+1}\right)_{n \in \mathbb Z} \ .
$$
Then $\rho_0 \circ \hat{L} = \hat{L} \circ \sigma$, which implies that $\rho_0(G_k) = G_k$ and hence 
$$\rho_0 (G) = G .
$$ 

For every element $f$ of $G$ there is an $R > 0$ and elements $f^\pm \in \mathcal G(\pm)$ such that $f(s) = f^-(s)$ when $\pi(s) \leq -R$ and $f(s) = f^+(s)$ when $\pi(s) \geq R$. Set
$$
\mathbb L_{\pm}(f) := \mathbb L_\pm(f^\pm).
$$ 
This is well-defined, i.e. depends only on $f$, since $\psi^\pm_1\circ \pi f^\pm$ is unique up to elements of $\mathcal A_c(S,\pi)$. We define $G^+$ to be the set which besides the zero function consists of the elements $f \in G$ for which
\begin{itemize}
\item[(1)] $f(s) > 0$ for all $s \in S$, and
\item[(2)] $R_\pm(e^{-\mathbb L_\pm(f)\pi}f^{\pm})$ are strictly  positive on $\partial S_\pm$.
\end{itemize}

\begin{lemma}\label{01-09-21} The pair $(G,G^+)$ has the following properties.
\begin{itemize}
\item[(0)] $G^+ + G^+ = G^+$.
\item[(1)] $G^+ \cap (-G^+) = \{0\}$.
\item[(2)] $G = G^+ - G^+$.
\item[(3)] $f\in G, \ n \in \mathbb N, \ nf \in G^+ \Rightarrow f \in G^+$.
\item[(4)] $(G,G^+)$ has the strong Riesz interpolation property.
\end{itemize}
\end{lemma}
\begin{proof} (1) and (3) are obvious, and (0) is proved as in the proof of Lemma \ref{09-01-23x}. (2): Let $f \in G$. Then $f \in G_k+ G_{00}$ for some $k \in \mathbb N$, and we can write
$$
f = f^-\psi^-_k\circ \pi + f_0\psi^0_k\circ \pi + f^+\psi^0_k \circ \pi + g
$$
where $f^\pm \in \mathcal G(\pm), \ f_0 \in \hat{L}(\bigoplus_{\mathbb Z} H)$ and $g \in G_{00}$. By definition of $\mathcal G(\pm)$ there are $m \in \mathbb N$ and $M > 0$ such that
\begin{equation}\label{19-07-23a}
f^+(s) \leq e^{m\pi(s)} \ \ \forall s \in \pi^{-1}([M,\infty[) \ 
\end{equation}
and

\begin{equation}\label{19-07-23}
f^-(s) \leq e^{-m\pi(s)} \ \ \forall s \in \pi^{-1}(]-\infty, -M]) .
\end{equation}
Choose $n \in \mathbb N$ such that $n \geq m, \ -n < \mathbb L_-(f)$ and $n > \mathbb L_+(f)$, and set
\begin{equation}\label{19-07-23b}
h:= e^{-n\pi}\psi_k^- \circ \pi + \psi^0_k \circ \pi + e^{n\pi}\psi^+_k \circ \pi .
\end{equation}
Set $c:= \psi^+_{2k+1} \circ \pi - \psi^+_1 \circ \pi$ and note that $c \in \mathcal C$ by (f) of Lemma \ref{helsingor}. Since $\psi^+_{2k+1} \psi^+_k = \psi^+_k$ we have 
$$
 e^{n\pi}\psi^+_k \circ \pi = \left(e^{n\pi} c +  e^{n\pi}\psi^+_1 \circ \pi\right) \psi^+_k \circ \pi \in \mathcal G(+)\psi^+_k \circ \pi .
 $$
Similarly, 
$$
 e^{-n\pi}\psi_k^- \circ \pi \in \mathcal G(-)\psi^-_k\circ \pi .
 $$
 Since $ \psi^0_k \circ \pi = L(\theta(u)) \psi^0_k \circ \pi$ we conclude that $h \in G$. By using that $f_0 \psi^0_k \circ \pi$ and $g$ are compactly supported in combination with the upper bounds \eqref{19-07-23} and \eqref{19-07-23a} we find a $K \in \mathbb N$ such that
$$
f(s) + 1 \leq K h(s) \ \ \forall s \in S .
$$
Note that $\mathbb L_\pm( K h) = \mathbb L_\pm(Kh -f) = \pm n$ and that
$$
R_\pm(e^{-\mathbb L_\pm(K h-f)\pi}(Kh-f)^\pm) =  R_\pm(e^{\mp n\pi} Kh^\pm) = K > 0.
$$
Since $Kh(s) - f(s) \geq 1$ and $Kh(s) \geq K$ for all $s \in S$ it follows that $Kh-f$ and $Kh$ are both in $G^+$. Thus $f =  Kh -(Kh-f) \in G^+- G^+$.

(4): Let $f_i < g_j, \ i,j \in \{1,2\}$, in $(G,G^+)$
. Since $\theta(H)$ has the Riesz interpolation property for the strict order by Lemma 3.1 in \cite{EHS}, there is $h_0 \in \theta(H)$ such that $f_i(x) < h_0(x) < g_j(x)$ for all $i,j$ and all $x \in \pi^{-1}(0)$. Note that there is a $R \in \mathbb N$ such that $f_i(s) =f^+_i(s), \ g_j(s) = g^+_j(s)$ for all $s \in S$ with $\pi(s) \geq R$, and   $f_i(s) =f^-_i(s), \ g_j(s) = g^-_j(s)$ for all $s \in S$ with $\pi(s) \leq -R$ and all $i,j$. Then $f^\pm_i < g^\pm_j$ in $(\mathcal G(\pm), \mathcal G(\pm)^+_R)$. Since $(\mathcal G(\pm), \mathcal G(\pm)^+_R)$ are strong Riesz groups by Lemma \ref{skovshoved9} there are elements $h^\pm \in \mathcal G(\pm)$ such that $f^\pm_i < h^\pm < g^\pm_j$ in $(\mathcal G(\pm), \mathcal G(\pm)^+_R)$.  Having $h^\pm$ and $R$ we use the statement (2) of Lemma 4.4 in \cite{ET} to find $H\in \mathcal A(S,\pi)$ such that $H(s) = h^-(s)$ when $\pi(s) \leq -R$, $ H(s) = h^+(s)$ when $\pi(s) \geq R$, $H(x) = L(h_0)(x)$ when $x \in \pi^{-1}(0)$, and
$
f_i(s) < H(s) < g_j(s)$
for all $s \in S$ and all $i,j$. Set
$$
H' := H\psi^-_k \circ \pi  + L(h_0) \psi^0_k\circ \pi  + H \psi^+_k\circ \pi   .
$$
If $k$ is large enough we have 
$$
f_i(s) < H'(s) < g_j(s)
$$
for all $s \in S$ and all $i,j$. Set 
$$
H'' := H' - L(h_0)\psi^0_k \circ \pi  - h^-\psi^-_k\circ \pi  - h^+\psi^+_k\circ \pi   ,
$$
and note that $\sup H'' \subseteq \pi^{-1}(]-R,R[ \backslash \{0\})$. Let $\delta > 0$ be given, smaller than $g_j(s) - H'(s)$ and $H'(s) -f_i(s)$ for all $i,j$ and all $s \in \pi^{-1}([-R,R])$.
By Property \ref{07-09-21} there is an element $g' \in G_{00}$ such that 
$$
\supp g'  \subseteq \pi^{-1}(]-R, R[ \backslash \{0\}) \ 
$$
and $\sup_{s \in S}\left|g'(s) - H''(s)\right| < \frac{\delta}{2}$. Set
$$
h := g' +  L(h_0)\psi^0_k \circ \pi + h^-\psi^-_k\circ \pi + h^+\psi^+_k\circ \pi ,
$$
and note that $h \in G$. Furthermore, $f_i(x) < h(x) < g_j(x)$ for all $i,j$ and all $x \in S$. Since $h(s) = h^+(s)$ for all $s \in S$ with $\pi(s)$ large enough, we have that $\mathbb L_+(h-f_i) = \mathbb L_+(h^+-f^+_i)$ and 
$$
R_+(e^{-\mathbb L_+(h-f_i)\pi}(h-f_i)^+) = R_+(e^{-\mathbb L_+(h^+-f^+_i)\pi}(h^+-f^+_i))
$$
is strictly positive since $f_i^+ < h^+_i$ in $(\mathcal G(\pm), \mathcal G(\pm)^+_R)$. It follows in the same way that 
$R_-(e^{-\mathbb L_-(h-f_i)\pi}(h-f_i)^-), \ R_+(e^{-\mathbb L_+(g_j-h)\pi}(g^j-h)^+)$, and $R_-(e^{-\mathbb L_-(g_j-h)\pi}(g^j-h)^-)$ 
are all strictly positive, and we conclude that $f_i < h < g_j$ in $(G,G^+)$. 

\end{proof}

Consider the subset $\Gamma$ of $\left( \bigoplus_\mathbb Z H\right) \oplus G$ consisting of the elements $(\xi, g)  \in \left( \bigoplus_\mathbb Z H\right) \oplus G$ with the property that there is an $\epsilon > 0$ such that
\begin{equation}\label{06-09-21d}
\hat{L}(\xi)(s) = g(s) \ \ \forall s \in \pi^{-1}(]-\epsilon,\epsilon[) \ .
\end{equation}
$\Gamma$ is a subgroup of $\left(\bigoplus_\mathbb Z H\right) \oplus G$.

\begin{lemma}\label{aabenraa} The projection $\Gamma \to G$ is surjective.
\end{lemma}
\begin{proof} Let $g \in G$. By definition of $G$ there is an element $\xi  \in \bigoplus_{\mathbb Z}H$ and $k \in \mathbb N$ such that $g(s) = \hat{L}(\xi)(s)$ for $s \in \pi^{-1}(\left]-\frac{1}{2k},\frac{1}{2k}\right[)$.
\end{proof}

Set
$$
\Gamma^+ := \left\{  (\xi, g)  \in \Gamma : \ g \in G^+ \backslash \{0\} \right\} \cup \{0\} \ .
$$
By combining Lemma \ref{aabenraa} and Lemma \ref{01-09-21} above with Lemma 3.2 in \cite{EHS} we conclude that $(\Gamma,\Gamma^+)$ is a dimension group; a strong Riesz group, in fact.

 \subsubsection{Additional properties of the dimension group} 
 
Since $\hat{L} \circ \sigma = \rho_0 \circ \hat{L}$ we can define $\rho \in \Aut \Gamma$ by
$$
\rho := \sigma \oplus \rho_0.
$$
Note that $\rho\left(\Gamma^+\right) = \Gamma^+$, i.e. $\rho \in \Aut(\Gamma,\Gamma^+)$. Given any element $h \in H$ we denote by $h^{(0)}\in \oplus_\mathbb Z H$ the element given by $h^{(0)}_0 = h$ and $h^{(0)}_n = 0$ when $n \neq 0$. Applied to the element $u \in H$ we get $u^{(0)}\in \oplus_\mathbb Z H$. Note that the constant function $1$ is in $G^+$. Set 
$$
v := (u^{(0)},1)
$$
and note that $v \in \Gamma^+$.
 Define $\Sigma : \oplus_\mathbb Z H \to H$ by
$$
\Sigma((h_n)_{n \in \mathbb Z}) = \sum_{n \in \mathbb Z} h_n .
$$

\begin{lemma}\label{aabenraa1} $(\Gamma,\Gamma^+)$ has large denominators; that is for all $x \in \Gamma^+$ and $m \in \mathbb N$ there is an element $y \in \Gamma^+$ and an $n \in \mathbb N$ such that $my \leq x \leq ny$ in $(\Gamma,\Gamma^+)$. 
\end{lemma}
\begin{proof} 

Let $ x=(\xi,g) \in \Gamma^+\backslash \{0\}$ and $m \in \mathbb N$ be given. Then 
$$
\theta(\Sigma(\xi))(x) = \hat{L}(\xi)(x) = g(x) > 0
$$
for all $x \in \pi^{-1}(0)$, implying that $\Sigma(\xi) \in H^+\backslash \{0\}$. Since $\theta(H)$ is dense in $\Aff \pi^{-1}(0)$ there is an element $b\in H^+$ such that $ mb< \Sigma(\xi)  < n b$ for some $n \in \mathbb N, \ n > m+2$. Since $L(\theta(\Sigma(\xi)))$ agrees with $g$ on $\pi^{-1}(0)$, there is a compact neighborhood $U$ of $0$ in $\mathbb R$ such that
$$
mL(\theta(b))(s) < g(s) < n L(\theta(b))(s)
$$
for all $s \in \pi^{-1}(U)$. There is also a $K \in \mathbb N$ such that $U \subseteq \ ]-K,K[$, $g(s) = g^-(s), \ s \leq -K$, and $g(s) = g^+(s),\ s \geq K$. It follows from (2) of Lemma 4.4 in \cite{ET} that there is an element $a \in \mathcal A(S,\pi)$ such that 
\begin{itemize}
\item[{}] $\frac{1}{n} g(s) < a(s) < \frac{1}{m} g(s) \ \ \forall s \in S$,
\item[{}] $a(s) = L(\theta(b))(s)$ for all $s \in \pi^{-1}(U)$,
\item[{}]  $a(s) = \frac{1}{ m+1}g^-(s)$ for $s \leq -K$, and
\item[{}] $a(s) = \frac{1}{m+1}g^+(s)$ for all $s \geq K$.
\end{itemize}
Choose $k \in \mathbb N$ so large that $[-\frac{1}{k},\frac{1}{k}] \subseteq U$ and note that 
$$
a(s) = L(\theta(b))(s)\psi^0_k\circ \pi(s) + a(s)\psi^+_k\circ\pi(s)  +a(s)\psi^-_k\circ \pi(s)  .
$$
Then the function
$$
a' := a - L(\theta(b))\psi^0_k\circ \pi - \frac{1}{m+1}g^+\psi^+_k\circ \pi - \frac{1}{m+1}g^-\psi^-_k\circ \pi
$$
is supported in 
$\pi^{-1}(]-K,K[ \backslash \{0\})$. Let $\delta > 0$ be smaller than $\frac{1}{m}g(s) - a(s)$ and $a(s) - \frac{1}{n}g(s)$ for all $s \in \pi^{-1}([-K,K])$. By Property \ref{07-09-21} we can find $c \in G_{00}$ such that $\supp c \subseteq \pi^{-1}(]-K,K[ \backslash \{0\})$ and $\left|c(s) -a'(s)\right| < \delta$ for all $s \in S$. Note that $\frac{1}{m+1}g^\pm \in \mathcal G(\pm)$ by (c) in Lemma \ref{helsingor}. Thus
$$
g_1 := c  +L(\theta(b))\psi^0_k\circ \pi + \frac{1}{m+1}g^+\psi^+_k\circ \pi + \frac{1}{m+1}g^-\psi^-_k\circ \pi \ \in \ G
$$
and $mg_1(s) < g(s) < ng_1(s)$ for all $s \in S$. Since $\mathbb L_\pm(ng_1 - g) = \mathbb L_\pm(g)$  and 
$$
R_\pm\left(e^{-\mathbb L_\pm(g)\pi} (ng_1-g)^\pm\right) = (\frac{n}{m+1} -1)R_\pm\left( e^{-\mathbb L_\pm(g)\pi}g^{\pm}\right)
$$
is strictly positive we find that $ng_1 > g$ in $(G,G^+)$. It follows in the same way that $g >mg_1$ in $(G,G^+)$, and we conclude that $m(b^{(0)},g_1)\leq (\xi,g) \leq n(b^{(0)},g_1)$ in $(\Gamma,\Gamma^+)$. 
\end{proof}

\begin{lemma}\label{aabenraa2} The only order ideals $I$ in $(\Gamma,\Gamma^+)$ such that $\rho(I) = I$ are $I = \{0\}$ and $I = \Gamma$.
\end{lemma}
\begin{proof} Recall that an order ideal $I$ in $(\Gamma,\Gamma^+)$ is a subgroup of $\Gamma$ such that
 \begin{itemize}
\item[(a)] $I= I \cap \Gamma^+ -  I \cap \Gamma^+$, and
\item[(b)] when $ 0 \leq y \leq x$ in $(\Gamma,\Gamma^+)$ and $x \in I$, then $y \in I$.
\end{itemize}
Let $I$ be a non-zero order ideal such that $\rho(I) = I$. Since $I \cap \Gamma^+ \neq \{0\}$ there is an element $g \in G^+ \backslash \{0\}$ and an element $\xi \in \bigoplus_{\mathbb Z}H$ such that $(\xi,g) \in I$. Set $h := \Sigma(\xi)$. Since $L(\theta(h))(x) = \hat{L}(\xi)(x) = g(x) > 0$ for all $x \in \pi^{-1}(0)$ there is a $k \in \mathbb N$ such that $L(\theta(h))(s) > 0$ for all $s \in \pi^{-1}([-\frac{1}{k},\frac{1}{k}])$. Set 
$$
g_1 := L(\theta(h))\psi^0_k \circ \pi + e^{\mathbb L_-(g)\pi}\psi^-_k\circ \pi  +  e^{\mathbb L_+(g)\pi}\psi^+_k\circ \pi .
$$
By the arguments used for the function \eqref{19-07-23b} it follows that $g_1 \in G$. (Alternatively one can use that $\psi^\pm_k\psi^\pm_1 = \psi^\pm_k$ if $k \geq 2$.) Note that there is a natural number $ K \in \mathbb N$ such that 
$$
0 < g_1(s)  < K g(s)  \  \ \ \forall s \in S   
$$ 
and
$$
K R_\pm(e^{-\mathbb L_\pm(g) \pi}g^\pm) \geq 2 .
$$
Furthermore,  
$$
R_\pm(e^{-\mathbb L_\pm(g)\pi}g_1^\pm)  =1 > 0
$$
and
$$
R_\pm(e^{-\mathbb L_\pm(Kg-g_1)\pi}(Kg-g_1)^\pm) = R_+(e^{-\mathbb L_\pm(g)\pi}(Kg-g_1)^\pm) \geq 2-1 =   1 > 0 . 
$$
Hence $0 <g_1 < Kg$ in $(G,G^+)$, implying that $({h}^{(0)}, g_1) \in I \cap \Gamma^+$. Note that
$$
\rho_0^{ l}(g_1) =  e^{- l \pi} L(\theta(h))\psi^0_k \circ \pi +  e^{(\mathbb L_-(g) - l)\pi}\psi^-_k\circ \pi +  e^{(\mathbb L_+(g)-  l)\pi}\psi^+_k\circ \pi
$$
for all $l \in \mathbb Z$. Consider an arbitrary element $(\xi',f) \in \Gamma^+ \backslash \{0\}$. We can then find $l_1,l_2 \in \mathbb Z$ such that $l_1 < l_2$, $\mathbb L_+(g) - l_1 > \mathbb L_+(f)$ and $\mathbb L_-(g) - l_2 < \mathbb L_-(f)$. Set
$$
H:= \rho_0^{l_1}(g_1)+ \rho_0^{ l_2}(g_1) 
$$
and $\xi'' := \sigma^{l_1}(h^{(0)}) + \sigma^{l_2}(h^{(0)}) \in \oplus_\mathbb Z H$, and note that $(\xi'',H) \in I\cap \Gamma^+$. Since $\mathbb L_+(g) - l_1 > \mathbb L_+(f)$ and $\mathbb L_-(g) - l_2 < \mathbb L_-(f)$ and because $H$ is strictly positive there is an $M \in \mathbb N$ such that
$$
f(s) < MH(s) \ \ \forall s \in S.
$$
Furthermore, $\mathbb L_+( MH - f) = \mathbb L_+(g) - l_1$ and 
$$
R_+\left(e^{-\mathbb L_+(MH -f)\pi}(MH-f)^+\right) = M > 0 .
$$
Similarly, $\mathbb L_-( MH - f) = \mathbb L_-(g) - l_2$ and
$$
R_-\left(e^{-\mathbb L_-(MH -f)\pi}(MH-f)^-\right) = M > 0 .
$$
It follows that $f < M H$ in $(G,G^+)$ and hence that $(\xi',f) \leq M(\xi'',H)$ in $(\Gamma,\Gamma^+)$. It follows that $(\xi',f) \in I$ and we conclude therefore that $I = \Gamma$.
\end{proof}

We denote in the sequel by $\mathcal A_\mathbb R(S,\pi)$ the real Banach space consisting of the elements of $\mathcal A(S,\pi)$ that have a limit at infinity.

\begin{lemma}\label{01-09-21e} 
\begin{itemize}
\item[(a)] Let $f \in \mathcal A_\mathbb R(S,\pi)$ and let $\epsilon > 0$ be given. There is an element $g\in G$ such that $\sup_{s \in S} |f(s)- g(s)| \leq \epsilon $.
\item[(b)]  Let $f \in \mathcal A_c(S,\pi)$ and let $\epsilon > 0$ be given. There is a $g \in G \cap \mathcal A_c(S,\pi)$ such that $\sup_{s \in S} |f(s)- g(s)| \leq \epsilon $ and  $\sup_{s \in S} |e^{-\pi(s)}f(s)- e^{-\pi(s)}g(s)| \leq \epsilon $.
\end{itemize}
\end{lemma}
\begin{proof} (a) An initial approximation gives us an element $f_1 \in \mathcal A_c(S,\pi)$ and a real number $r \in \mathbb R$ such that 
$$
\sup_{s \in S} |f(s)- f_1(s) - r| \leq \frac{\epsilon}{2} \ .
$$
Choose $q \in \mathbb Q$ such that $|q-r| < \frac{\epsilon}{6}$ and choose an element $h \in H$ such that $|\theta(h)(x) - f_1(x) -r| < \frac{\epsilon}{6}$ for all $x \in \pi^{-1}(0)$. There is a $k \in \mathbb N$, $k \geq 2$, such that $\left|L(\theta(h))(s) -f_1(s) -r\right| < \frac{\epsilon}{6}$ for all $s \in \pi^{-1}([-\frac{1}{k},\frac{1}{k}])$. Since $f_1\psi^+_k\circ \pi + f_1\psi^-_k\circ \pi$ is compactly supported in $S \backslash \pi^{-1}(0)$ it follows from Property \ref{07-09-21} that there is an element $g' \in G_{00}$ such that
$$
\sup_{s \in S}\left| g'(s) - f_1(s)\psi^+_k\circ \pi(s) - f_1(s)\psi^-_k\circ \pi(s)\right| \leq \frac{\epsilon}{6} \ .
$$
Set
$$
g := L(\theta(h))\psi^0_k \circ \pi + q\psi^+_k\circ \pi  + q\psi^-_k\circ \pi  + g' .
$$
Note that $g \in G$. It is easy to see that $g$ has the desired property. 

(b) Choose $N \in \mathbb N$ such that $\supp f \subseteq ]-N,N[$. Repeat then the proof of (a) with $\epsilon$ replaced by $e^{-N}\epsilon$ and $f_1$ replaced by $f$. Take $r=q = 0$ and choose $g'$ from Property \ref{07-09-21} such that $\supp g' \subseteq \pi^{-1}(]-N,N[\backslash \{0\})$.
\end{proof}

For each $s \in S$ we can define a positive homomorphism $\ev_s : \Gamma \to \mathbb R$ such that
$$
\ev_s(\xi,g) := g(s) \ .
$$ 
Then $\ev_s(v) = 1$, and $\ev_s \circ \rho = e^{-\pi(s)} \ev_s$.

\begin{lemma}\label{27-08-21x}
Let $\phi : \Gamma \to \mathbb R$ be a positive homomorphism with the properties that $\phi(v) =1$ and $\phi \circ \rho = t \phi$ for some $t > 0$. Set $\beta := -\log t$. There is an element $s\in \pi^{-1}(\beta)$ such that $\phi = \ev_s$.
\end{lemma}
\begin{proof} The projection $p : \Gamma \to G$ is surjective by Lemma \ref{aabenraa}. Assume that $(\xi,g) \in \Gamma \backslash \{0\}$ and $p(\xi,g) = g = 0$. Since $(\Gamma,\Gamma^+)$ has large denominators by Lemma \ref{aabenraa1} there is for each $n \in \mathbb N$ an element $(\xi_n,g_n) \in \Gamma^+\backslash \{0\}$ such that $n(\xi_n,g_n) \leq v$ in $(\Gamma,\Gamma^+)$. Then $\pm (\xi,g) \leq (\xi_n,g_n)$ in $(\Gamma,\Gamma^+)$ and hence $\pm \phi(\xi,g) \leq \phi(\xi_n,g_n) \leq \frac{1}{n}$. It follows that $\phi(\xi,g) = 0$ and we conclude that there is a homomorphism $\phi' : G \to \mathbb R$ such that $\phi' \circ p = \phi$. 

Let $g \in G$ and assume that $g(s) \geq 0$ for all $s \in S$.  Let $n \in \mathbb N, \ n \geq 2$. Since $\theta(H)$ is a dense in $\Aff \pi^{-1}(0)$ there is an $\eta \in H$ such that $0 < \theta(\eta)(x) < \frac{1}{n}$ for all $x \in \pi^{-1}(0)$. Set
$$
h^\pm := \frac{1}{n}e^{L_\pm(g)\pi}\psi^\pm_1 \circ \pi .
$$
It follows from (a) and (c) of Lemma \ref{helsingor} and the definitions of $A(\pm)$ and $G(\pm)$ that $h^{\pm} \in G(\pm) \subseteq \mathcal G(\pm)$ and hence that
$$
h:=  h^-\psi^-_k \circ \pi + L(\theta(\eta))\psi^0_k\circ \pi + h^+\psi^+_k \circ \pi 
$$
is an element of $G$ regardless of which $k \in \mathbb N$ we choose. Since $L(\theta(\eta))(x) = \theta(\eta)(x) > 0$ for all $x \in \pi^{-1}(0)$ we can choose $k$ so large that $0 < L(\theta(\eta))(s) < \frac{1}{n}$ for all $s \in \pi^{-1}([-\frac{1}{k}, \frac{1}{k}])$. Since 
$$
(n-1) h^\pm \psi^\pm_k\circ \pi \leq e^{L_\pm(g)\pi} 
$$
in $(G,G^+)$
and
$$
(n-1)  L(\theta(\eta))\psi^0_k\circ \pi  \leq 1
$$
in $(G,G^+)$ it follows that
\begin{align*}
&\phi'(h) \leq \frac{1}{n-1} + \frac{1}{n-1}\phi'(\rho_0^{-L_+(g)}(1))+ \frac{1}{n-1}\phi'(\rho_0^{-L_-(g)}(1)) \\
&=  \frac{1}{n-1}\left( 1 + t^{-\mathbb L_-(g)} + t^{-\mathbb L_+(g)}\right) .
\end{align*}
Now note that since $g \geq 0$ it follows that $g+h \in G^+$. Since $\phi$ is positive this implies that
$$
0 \leq \phi'(g+h) \leq \phi'(g) + \frac{1}{n-1}\left( 1 + t^{-\mathbb L_-(g)} + t^{-\mathbb L_+(g)}\right) .
$$
Since this holds for all $n \geq 2$ we conclude that $\phi'(g) \geq 0$. Thus $\phi'$ is positive in the mundane sense that 
$$
g\in G, \ g(s) \geq 0 \ \forall s \in S \Rightarrow \phi'(g) \geq 0.
$$

Let $g \in G$ and $n,m \in \mathbb N$ satisfy $|g(s)| < \frac{n}{m}$ for all $s \in S$. Then $-n <mg(s) < n$ for all $s \in S$ and since $\phi'(1) = 1$ this leads to the conclusion that $|\phi'(g)| \leq \frac{n}{m}$. Combined with (a) of Lemma \ref{01-09-21e} it follows from the last estimate that $\phi'$ extends by continuity from $G \cap \mathcal A_\mathbb R(S,\pi)$ to a linear map $\phi' : \mathcal A_\mathbb R(S,\pi)  \to \mathbb R$ such that $\left|\phi'(f)\right| \leq \sup_{s \in S} |f(s)|$ for all $f \in \mathcal A_\mathbb R(S,\pi)$. Using the Hahn-Banach theorem we extend $\phi'$ in a norm-preserving way to the space of all continuous real-valued functions on $S$ with a limit at infinity. Since $\phi'(1) = 1$ the extension is positive. It follows from the Riesz representation theorem that there is a bounded Borel measure $m$ on $S$ such that
$$
\phi'(f) = \int_S f(s) \ \mathrm{d} m(s)
$$
for all $f \in \mathcal A_{0}(S,\pi)$, where $\mathcal A_0(S,\pi)$ denotes the space of elements in $\mathcal A_\mathbb R(S,\pi)$ that vanish at infinity. 
Let $C_c(\mathbb R)$ denote the set of continuous real-valued compactly supported functions on $\mathbb R$ and note that $C_c(\mathbb R)$ is mapped into $\mathcal A_c(S,\pi)$ by the formula $F \mapsto F \circ \pi$. Since $\phi' \circ \rho_0 = t\phi'$ it follows from (b) of Lemma \ref{01-09-21e} that $\phi'(e^{-\pi}h) = t \phi'(h)$ for all $h \in \mathcal A_c(S,\pi)$. In particular,
$$
 \int_\mathbb R e^{-y}F(y) \ \mathrm{d}m \circ \pi^{-1}(y)  =  t\int_\mathbb R F(y) \ \mathrm{d}m \circ \pi^{-1}(y)  \ \ \forall F \in C_c(\mathbb R)  .
 $$
It follows from this that $m \circ \pi^{-1}$ is concentrated at the point $\beta = -\log t$ and hence that $m$ is concentrated on $\pi^{-1}(\beta)$. We can therefore define a linear functional $\phi'' : \Aff \pi^{-1}(\beta) \to \mathbb R$ by
$$
\phi''(f) := \phi'(\hat{f}) = \int_S \hat{f}(x) \ \mathrm{d} m(x) \ ,
$$
where $\hat{f} \in \mathcal A_0(S,\pi)$ is any element with $\hat{f}|_{\pi^{-1}(\beta)} = f$. Such an $\hat{f}$ exists by (1) in Lemma 4.4 of \cite{ET}. If $f \geq 0$  in $\Aff \pi^{-1}(0)$ and $\epsilon > 0$ are given we can choose $\hat{f}$ such that $\hat{f} \geq - \epsilon$ and we see therefore that $\phi''$ is a positive linear functional. Since every state of $\Aff \pi^{-1}(\beta)$ is given by evaluation at a point in $\pi^{-1}(\beta)$ it follows in this way that there is an $s\in \pi^{-1}(\beta)$ and a number $\lambda \geq 0$ such that 
\begin{equation}\label{01-09-21h}
\phi'(g) = \lambda g(s)
\end{equation}
for all $g  \in \mathcal A_0(S,\pi)$. In particular, this conclusion holds for all $g \in G \cap \mathcal A_0(S,\pi)$. A general element $f \in G$ can be write as a sum
$$
f = f_- + f_0 + f_+ ,
$$
where $f_\pm, f_0 \in G$, $f_0$ has compact support and there are natural numbers $n_{\pm} \in \mathbb N$ such that $e^{n_-\pi}f_- \in \mathcal A_0(S,\pi)$ and $e^{-n_+ \pi}f_+ \in \mathcal A_0(S,\pi)$. Then $\phi'(f_0) = \lambda f_0(s)$,
\begin{align*}
&\phi'(f_-) = \phi'(\rho_0^{n_-}(e^{n_-\pi}f_-)) = t^{n_-}\phi'(e^{n_-\pi}f_-) \\
&=  t^{n_-}\lambda e^{n_-\pi(s)}f_-(s) = \lambda f_-(s),
\end{align*}
and similarly, $\phi'(f_+) = \lambda f_+(s)$.
It follows that $\phi(f) = \lambda f(s)$. Inserting $f=1$ we find that $\lambda = 1$ and the proof is complete.
\end{proof}

\subsection{The crossed product $C^*$-algebra}

It follows from \cite{EHS} and \cite{E1} that there is a stable AF-algebra $A$ such that $(K_0(A),K_0(A)^+) = (G, G^+)$ and an automorphism $\alpha \in \Aut A$ such that $\alpha_* = \rho$. Furthermore, there is a projection $e \in A$ such that $[e] = v$ in $K_0(A)$. We can then consider the crossed product $A \rtimes_\alpha \mathbb Z$ and the flow $\hat{\alpha}^e$ on the corner $e(A\rtimes_\alpha \mathbb Z)e$. However, there is a freedom in the choice of the lift $\alpha \in \Aut A$ of $\rho \in \Aut (G,G^+)$ which will be crucial in order to fix the Elliott invariant of $e(A\rtimes_\alpha\mathbb Z)e$. We return to this point in Section \ref{Elliott}. In this and the following section we check that the KMS bundle of $\hat{\alpha}^e$ is isomorphic $(S,\pi)$ and that the set of $\KMS_\infty$ states and $\KMS_{-\infty}$ states of $\hat{\alpha}^e$ are homeomorphic to $D_+$ and $D_-$, respectively, and for this any lift will work.

\begin{lemma}\label{21-06-23x} $A\rtimes_\alpha \mathbb Z$ is simple.
\end{lemma}
\begin{proof} It follows from Lemma \ref{aabenraa2} and the bijective correspondence between ideals of $A$ and order ideals in $(\Gamma,\Gamma^+)$ that there are no non-trivial ideals of $A$ left invariant by $\alpha$. Combined with the observation that $\alpha^k_* = \rho^k$ is non-trivial when $k \neq 0$ it follows that $\alpha^k$ is properly outer in the sense of \cite{E2} for all $k \neq 0$. It follows therefore from Theorem 3.2 of \cite{E2} that $A\rtimes_\alpha \mathbb Z$ is simple.
\end{proof} 

It follows from Lemma \ref{21-06-23x} that $e$ is a full projection in $A \rtimes_\alpha \mathbb Z$. Hence Theorem 4.1 in \cite{Th1} shows that for every $\beta \in \mathbb R$ there is an affine bijection $\omega \mapsto \tilde{\omega}$ from the set of $\beta$-KMS states $\omega$ for $\hat{\alpha}^e$ onto the set of $\beta$-KMS weights $\tilde{\omega}$ for $\hat{\alpha}$ with the property that $\tilde{\omega}(e) = 1$. Using how the KMS-weights for $\hat{\alpha}$ were described by Vigand Pedersen in Theorem 5.1 of \cite{VP} (and again by the author in Lemma 3.1 of \cite{Th3}) we obtain the following

\begin{lemma}\label{06-08-21} Let $\hat{\alpha}^e$ be the restriction to $e(A\rtimes_\alpha \mathbb Z)e$ of the dual flow on $A\rtimes_\alpha\mathbb Z$. Let $P : A \rtimes_\alpha \mathbb Z \to A$ be the canonical conditional expectation. For each $\beta \in \mathbb R$ the map 
$$
\tau \mapsto \tau \circ P|_{e(A \rtimes_\alpha \mathbb Z)e}
$$ 
is a bijection from the set of lower semi-continuous traces $\tau$ on $A$ that satisfy 
\begin{itemize}
\item $\tau \circ \alpha = e^{-\beta} \tau$, and
\item $\tau(e) =1$,
\end{itemize}
onto the simplex of $\beta$-KMS states for $\hat{\alpha}^e$.
\end{lemma}

Since $A$ is an AF-algebra the map $\tau \mapsto \tau_*$ is a bijection from the set of lower semi-continuous traces $\tau$ on $A$ onto the set $\Hom^+(K_0(A),\mathbb R)$ of non-zero positive homomorphisms $\phi :  K_0(A) \to \mathbb R$.
In this way Lemma \ref{06-08-21} has the following

\begin{cor}\label{06-08-21a} For each $\beta \in \mathbb R$ the map $\omega \mapsto (\tilde{\omega}|_A)_*$
is an affine bijection from the set of $\beta$-KMS states for $\hat{\alpha}^e$ on $e(A \rtimes_\alpha \mathbb Z)e$ onto the set of positive homomorphisms $\phi \in \Hom^+(K_0(A),\mathbb R)$ that satisfy 
\begin{itemize} 
\item $\phi \circ \rho = e^{-\beta} \phi$, and
\item $\phi([e]) =1$.
\end{itemize}
\end{cor} 

Another consequence of Lemma \ref{06-08-21} is the following

 \begin{cor}\label{10-01-23e} Let $\omega$ be a $\KMS_\infty$ or a $\KMS_{-\infty}$ state for $\hat{\alpha}^e$. There is a trace state $\tau$ on $eAe$ such that
$$
\omega = \tau \circ P|_{ e(A \rtimes_\alpha \mathbb Z)e}.
$$
\end{cor}
\begin{proof} A state of $ e(A \rtimes_\alpha \mathbb Z)e$ which is the weak* limit of a sequence of states that have the form $\tau \circ P$ for some trace state $\tau$ on $eAe$ must itself be of this form.
\end{proof}
 
Lemma \ref{06-08-21} and its two corollaries give the tools we shall use to describe both the $\KMS$ bundle of $\hat{\alpha}^e$ and the sets of $\KMS_\infty$ states and $\KMS_{-\infty}$ states for $\hat{\alpha}^e$ from the preceding K-theory constructions.

\subsection{Recovering $(S,\pi)$ and $D_\pm$}

\begin{lemma}\label{02-09-21d} The $\KMS$ bundle $(S^{\hat{\alpha}^e},\pi^{\hat{\alpha}^e})$ of $\hat{\alpha}^e$ is isomorphic to $(S,\pi)$.
\end{lemma}
\begin{proof} Let $\omega$ be a $\beta$-KMS state for $\hat{\alpha}^e$. By Corollary \ref{06-08-21a} $(\tilde{\omega}|_A)_*$ is a positive homomorphism $\Gamma \to \mathbb R$ to which Lemma \ref{27-08-21x} applies and gives us an element $s \in \pi^{-1}(\beta)$ such that $(\tilde{\omega}|_A)_* = \ev_s$. This results in a map $\Phi : S^{\hat{\alpha}^e} \to S$ such that $\Phi( \omega) = s$. Except for changes in the notation the proof of Lemma 4.19 in \cite{ET} can be repeated to show that $\Phi$ is an isomorphism of proper simplex bundles.
\end{proof}

\begin{prop}\label{21-06-23} The set $\KMS_\infty(\hat{\alpha}^e)$ of $\KMS_\infty$ states for $\hat{\alpha}^e$ is homeomorphic to $D_+$ and the set $\KMS_{-\infty}(\hat{\alpha}^e)$ of $\KMS_{-\infty}$ states for $\hat{\alpha}^e$ is homeomorphic to $D_-$.
\end{prop}

\begin{proof} Except for notational differences the $+\infty$-case and the $-\infty$-case are identical and we give here the details only in the $+\infty$-case. Let $\omega \in \KMS_\infty(\hat{\alpha}^e)$. There is then a sequence $\{\omega_i\}_{i=1}^\infty$ of states on $e(A \rtimes_\alpha \mathbb Z)e$ such that $\omega_i$ is a $\beta_i$-KMS state for $\hat{\alpha}^e$, $\lim_{i \to \infty} \beta_i = \infty$ and $\lim_{i \to \infty} \omega_i = \omega$ in the weak* topology. It follows from Lemma \ref{06-08-21} that the restriction of $\omega_i$ to $eAe$ is a trace state which is the restriction to $eAe$ of a lower semi-continuous trace on $A$ with the properties specified in Lemma \ref{06-08-21}.
Note that we can and will identify $K_0(eAe)$ with the subgroup of $\Gamma$ consisting of the elements $x \in \Gamma$ for which there is an $l \in \mathbb N$ such that $-l v \leq x \leq lv$ in $(\Gamma,\Gamma^+)$. Thus, when $(\xi,g) \in K_0(eAe)$ there is an $l \in \mathbb N$ such that $|g(s)| \leq l$ for all $s \in S$. In particular, $R_+(g) \in C_\mathbb R(\partial S_+)$ is defined for all $(\xi,g) \in K_0(eAe)$. It follows from Lemma \ref{27-08-21x} that there is an element $s_i \in S$ such that
$$
({\omega_i|_{eAe}})_*(\xi,g) = g(s_i)
$$
when $(\xi,g) \in K_0(eAe)$. Let $\eta$ be a condensation point of $\{s_i\}$ in the Stone-\v{C}ech compactification of $S$. Then $\eta \in \partial S_+$ since $\lim_{i \to \infty} \pi(s_i) = \lim_{i \to \infty} \beta_i =  \infty$. Since $K_0(eAe)$ is countable there is a subsequence $\{s_{i_j}\}$ of $\{s_i\}$ such that $\lim_{j \to \infty} g(s_{i_j}) = R_+(g)(\eta)$ for all $(\xi,g) \in K_0(eAe)$. Since $\lim_{i \to \infty} \omega_i|_{eAe} = \omega|_{eAe}$ we find that
\begin{equation}\label{dragoer5}
\begin{split}
&\left(\omega|_{eAe}\right)_*(\xi,g) = \lim_{i \to \infty}\left(\omega_i|_{eAe}\right)_*(\xi,g) = \lim_{j \to \infty} g(s_{i_j}) = R_+(g)(\eta) 
\end{split}
\end{equation}
for all $(\xi,g) \in K_0(eAe)$. 

Let $(\xi,g) \in K_0(eAe)$. By definition of $G$ and $\mathcal G(+)$ there is an $N \in \mathbb N$ and an element $\overline{g} \in G(+)$ such that $g^+ = (1 - e^{-\pi})^{-N} \overline{g}$. Decompose $\overline{g}$ as in \eqref{holbaek3}. Then
$$
g^+ = \sum_{n \in \mathbb Z} (1-e^{-\pi})^{-N} \overline{g}_n e^{n \pi} .
$$
We claim that
\begin{obs}\label{dragoer10}
$\left(\omega|_{eAe}\right)_*(\xi,g) = R_+(\overline{g}_0)(\eta)$ for all $(\xi,g) \in K_0(eAe)$.
\end{obs}

To see this, choose $l \in \mathbb N$ such that $- l < g <l$ in $(G,G^+)$. Assume for a contradiction that $\mathbb L_+(g) > 0$. Then $\mathbb L_+(l-g) = \mathbb L_+(g) > 0$ and hence $R_+(e^{-\mathbb L_+(g)\pi}(l- g)^+)$ is strictly positive on $\partial S_+$. This implies that there is an $\epsilon > 0$ and a $T > 0$ such that 
$$
e^{-\mathbb L_+(g)\pi(s)}(l-g(s)) \geq \epsilon
$$
for all $s \in \pi^{-1}([T,\infty))$. Since $\lim_{\pi(s) \to \infty} e^{-\mathbb L_+(g)\pi(s)} = 0$ there is an $R > 0$ such that $\frac{\epsilon}{2} -e^{-\mathbb L_+(g)\pi(s)}g(s) \geq \epsilon$ for $\pi(s) \geq R$ and hence 
$$
g(s) \leq -e^{\mathbb L_+(g)\pi(s)}\frac{\epsilon}{2}
$$ 
for $\pi(s) \geq R$, contradicting that $-l < g$ in $(G,G^+)$. It follows that $\mathbb L_+(g) \leq 0$, and that
$$
\psi^+_1\circ \pi g^+ = c + \sum_{n\leq 0} (1-e^{-\pi})^{-N}\psi^+_1\circ \pi \overline{g}_n e^{n\pi} 
$$
for some $c \in \mathcal A_c(S,\pi)$. Since 
$$
R_+(c + \sum_{n\leq -1} (1-e^{-\pi})^{-N} \psi^+_1\circ \pi\overline{g}_n e^{n\pi}) = 0
$$
we conclude that $R_+(g) = R_+(\overline{g}_0)$. In this way Observation \ref{dragoer10} follows from \eqref{dragoer5}.

 By definition of $A(+)$ there is a unique $h \in C_\mathbb R(D_+)$ such that $R_+(\overline{g}_0) = h \circ r_+$ and hence
\begin{equation}\label{dragoer12}
\left(\omega|_{eAe}\right)_*(\xi,g) = h(r_+(\eta)) 
\end{equation}
by Observation \ref{dragoer10}. In this way the element $\omega \in \KMS_\infty(\hat{\alpha}^e)$ gives rise to an element $r_+(\eta) \in D_+$. To see that this element is unique, let $k \in C_\mathbb R(D_+)$ be an element such that $k \circ r_+\in C_+$. Then $\psi^+_1 \circ \pi l_+(k \circ r_+)\in G$, 
$$
(0,\psi^+_1 \circ \pi l_+(k \circ r_+)) \in K_0(eAe), 
$$
and \eqref{dragoer12} implies
\begin{equation}\label{dragoer11}
 (\omega|_{eAe})_*((0,\psi^+_1 \circ \pi l_+(k \circ r_+))) =k(r_+(\eta)).
\end{equation}
Since $C_+$ is dense in $r^*_*(C_\mathbb R(D_+))$ by construction, the set
$$
\left\{k \in C_\mathbb R(D_+) : \ k \circ r_+ \in C_+\right\}
$$
is dense in $C_\mathbb R(D_+)$ and hence \eqref{dragoer11} shows that $r_+(\eta) \in D_+$ is uniquely determined by $\omega$. We can therefore define a map 
$$
\Phi : \KMS_\infty(\hat{\alpha}^e) \to D_+
$$ 
such that
$$
\Phi(\omega) = r_+(\eta) .
$$ 
$\Phi$ is continuous: If $\lim_{n \to \infty} \omega_n = \omega$ in $\KMS_\infty(\hat{\alpha}^e)$ we observe that $\omega_n|_{eAe}$ and $\omega|_{eAe}$ are all trace states on $eAe$ by Corollary \ref{10-01-23e}. It follows therefore that
$$
\lim_{n \to \infty} (\omega_n|_{eAe})_*(x) = (\omega|_{eAe})_*(x)
$$ 
for all $x \in K_0(eAe)$, and then from \eqref{dragoer11} that $\lim_{n \to \infty} k(\Phi(\omega_n)) = k(\Phi(\omega))$ for all $k \in C_\mathbb R(D_+)$, implying that $\lim_{n \to \infty} \Phi(\omega_n) = \Phi(\omega)$ in $D_+$. 

$\Phi$ is injective: If $\omega_1,\omega_2\in \KMS_\infty(\hat{\alpha}^e)$ and $\Phi(\omega_1) = \Phi(\omega_2)$, it follows from \eqref{dragoer12} that $(\omega_1|_{eAe})_* =(\omega_2|_{eAe})_*$ which implies that $\omega_1|_{eAe} = \omega_2|_{eAe}$ since $eAe$ is an AF-algebra. It follows then from Corollary \ref{10-01-23e} that $\omega_1 = \omega_2$, implying that $\Phi$ is injective.

$\Phi$ is surjective: Let $y \in D_+$. Then $y = r_+(\eta)$ for some $\eta \in \partial S_+$ since $r_+$ is surjective by assumption. There is a net $\{s_i\}$ in $S$ such that $\lim_{i \to \infty} \pi(s_i) = \infty$ and $\lim_{i \to \infty} a( s_i) = R_+(a)(\eta)$ for all $a \in \mathcal A_b(S,\pi)$. Set $\beta_i = \pi(s_i)$. It follows from Corollary \ref{06-08-21a} that there is a $\beta_i$-KMS state $\omega_i$ for $\hat{\alpha}^e$ on $e(A\rtimes_\gamma\mathbb Z)e$ such that $(\omega_i|_{eAe})_*(\xi , g) = g(s_i)$ for all $(\xi,g) \in K_0(eAe)$. A condensation point of $\{\omega_i\}$ in the weak* topology is then a $\KMS_\infty$-state $\omega \in \KMS_\infty(\hat{\alpha}^e)$ such that $(\omega|_{eAe})_*(\xi,g) = R_+(g)(\eta)$ for all $(\xi,g) \in K_0(eAe)$. Hence $\Phi(\omega) = r_+(\eta) =y$.
\end{proof}

 To illustrate what we have proved so far and what is still lacking, choose for each $k \in \mathbb Z$ a non-empty metrizable Choquet simplex $\Delta_k$ and let $D_+$ and $D_-$ be a pair of non-empty compact metric spaces. The disjoint union
$$
S := \bigsqcup_{k \in \mathbb Z} \Delta_k,
$$
together with the map $\pi : S \to \mathbb R$ defined such that $\pi(s) = k$ for $s \in \Delta_k$, constitute a proper simplex bundle $(S,\pi)$. It follows from what we have proved above that there is an AF-algebra $A$, an automorphism $\alpha$ of $A$ and a projection $e \in A$ such that the $\KMS$ bundle of $\hat{\alpha}^e$ is isomorphic to $(S,\pi)$ while $\KMS_\infty(\hat{\alpha}^e)$ is homeomorphic to $D_+$ and  $\KMS_{-\infty}(\hat{\alpha}^e)$ is homeomorphic to $D_-$.\footnote{Some parts of the construction are in fact not really necessary for this; only for the arguments in the next section.} The algebra $e(A \rtimes_\alpha \mathbb Z)e$ on which the flow $\hat{\alpha}^e$ acts may vary with the choice of simplixes and spaces $D_\pm$, however. In the next section we fix this by using the recent powerful classification results for simple unital $C^*$-algebras. It will follow that no matter how we choose $\Delta_k, \ k\neq 0$, and $D_\pm$, we can arrange that $e(A \rtimes_\alpha \mathbb Z)e$ is isomorphic to any given infinite dimensional simple AF-algebra whose tracial state space is affinely homeomorphic to $\Delta_0$.

\subsection{The Elliott invariant - application of the classification theorem}\label{Elliott}

In this section we show how to choose the ingredients in the preceding in such a way that the Elliott invariant of $e(A \rtimes_\alpha \mathbb Z)e$ becomes isomorphic to that of the given AF-algebra $B$ in Theorem \ref{main}. We can then use classification results for simple unital $C^*$-algebras to deduce that with these choices $e(A \rtimes_\alpha \mathbb Z)e$ becomes isomorphic to $B$. This will be done basically in the same way as in \cite{Th3}, but we must check that the arguments also work with the new dimension group.

Since $(\Gamma,\Gamma^+)$ has large denominators by Lemma \ref{aabenraa1} we can use Lemma 3.4 of \cite{Th3} to arrange that $\alpha$ has the following additional properties:
 \begin{poem}\label{07-09-21c}
\begin{itemize}
\item[(A)] The restriction map 
$\mu \ \mapsto \ \mu|_A$
is a bijection from traces $\mu $ on $A \rtimes_{\alpha} \mathbb Z$ onto the $\alpha$-invariant traces on $A$, and
\item[(B)] $A \rtimes_{\alpha} \mathbb Z$ is $\mathcal Z$-stable; that is $(A \rtimes_{\alpha} \mathbb Z)\otimes \mathcal Z \simeq A \rtimes_{\alpha} \mathbb Z$ where $\mathcal Z$ denotes the Jiang-Su algebra, \cite{JS}.
\end{itemize} 
\end{poem}

Condition (A) is crucial in order to ensure that $e(A \rtimes_\alpha \mathbb Z)e$ does not have more trace states than $B$. The condition holds if and only if all traces on $A \rtimes_\alpha \mathbb Z$ are invariant under the dual flow $\hat{\alpha}$. Condition (B) is necessary to ensure that $e(A \rtimes_\alpha \mathbb Z)e$ is covered by the current classification results for simple $C^*$-algebras.

To simplify notation, set
$$
C := A\rtimes_\alpha \mathbb Z  \ .
$$
$C$ is simple by Lemma \ref{21-06-23x}.
It follows from the Pimsner-Voiculescu exact sequence, \cite{PV}, that we can identify $K_0(C)$, as a group, with the quotient
$$
\Gamma/(\id - \rho)(\Gamma)  ,
$$
in such a way that the map $\iota_* : K_0(A) \to K_0(C)$ induced by the inclusion $\iota : A \to C$ becomes the quotient map
$$
q : \Gamma \to  \Gamma/(\id - \rho)(\Gamma) .
$$
Define $S_0 : \Gamma \to H$ such that 
$$
S_0(\xi, g) = \Sigma(\xi ) .
$$

\begin{lemma}\label{01-09-21k} $\ker S_0 = (\id- \rho)(\Gamma)$.
\end{lemma}
\begin{proof} Since $ \id - \rho = (\id -\sigma) \oplus (\id - \rho_0)$ and $\Sigma \circ (\id - \sigma) = 0$, we find that $ (\id - \rho)(\Gamma) \subseteq \ker S_0$. Let $(\xi,g) \in \Gamma$, and assume that $S_0(\xi,g) =\Sigma(\xi) = 0$. By Lemma 4.6 of \cite{Th3} there is an element $\xi' \in \bigoplus_\mathbb Z H$ such that $(\id - \sigma)(\xi') = \xi$. By the definition of $\Gamma$ there is $\epsilon > 0$ such that $\hat{L}(\xi)$ and $g$ agree on $\pi^{-1}(]-\epsilon,\epsilon[)$, and we can therefore find a $k \geq \epsilon^{-1}$ such that 
$$
g = \hat{L}(\xi)\psi^0_k\circ \pi + g^-\psi^-_k\circ \pi + g^+\psi^+_k\circ \pi + g_0
$$ 
for some $g^\pm \in \mathcal G(\pm)$ and some $g_0 \in G_{00}$. By the definition of $\mathcal G(\pm)$ there are elements $f^\pm \in \mathcal G(\pm)$ such that $g^\pm = (\id -\rho_0)(f^\pm)$ and by \eqref{07-09-21e} there is an element $g' \in G_{00}$ such that $g_0 = (\id -\rho_0)(g')$. Define
$$
g'' :=  \hat{L}(\xi')\psi^0_k\circ \pi + f^-\psi^-_k\circ \pi + f^+\psi^+_k\circ \pi + g' \ \in \ G .
$$
Since 
$$
 (\id - \rho_0)(\hat{L}(\xi')\psi^0_k\circ \pi) = \hat{L}((\id - \sigma)(\xi'))\psi^0_k\circ \pi = \hat{L}(\xi)\psi^0_k\circ \pi , 
 $$
 it follows that $g = (\id - \rho_0)(g'')$ and hence that $(\xi,g) = (\id - \rho)( \xi',g'')$.
 \end{proof}

It follows from Lemma \ref{01-09-21k} that $S_0$ induces an isomorphism
$$
S : K_0(C) = \Gamma/(\id - \rho)(\Gamma)  \to H \ 
$$
such that $S \circ q = S_0$.

\begin{lemma}\label{02-09-21} $S(K_0(C)^+) = H^+$.
\end{lemma}
\begin{proof} Let $h \in H^+ \backslash \{0\}$. Since $L(\theta(h))(x) = \theta(h)(x)$ for all $x \in \pi^{-1}(0)$ there is a $k \geq 2$ so big that $L(\theta(h))(s) > 0$ for all $s \in \pi^{-1}([-\frac{1}{k},\frac{1}{k}])$. Define
$$
g :=  L(\theta(h))\psi^0_k\circ \pi + \psi^-_k \circ \pi + \psi^+_k \circ \pi  \ .
$$
Note that $g \in G^+$.
Then $(h^{(0)},g) \in \Gamma^+, \ q((h^{(0)},g)) \in K_0(C)^+$, and $S( q((h^{(0)},g))) = h$. Hence,
$H^+ \subseteq S(K_0(C)^+)$. Consider an element $\eta \in K_0(C)^+ \backslash \{0\}$ and write $\eta = q(\xi,g)$ for some $(\xi,g) \in \Gamma$. Let $x\in \pi^{-1}(0)$. Since $\ev_x \circ \rho = \ev_x$, there is an $\alpha$-invariant trace $\tau_x$ on $A$ such that ${\tau_x}_* = \ev_s$; see Lemma 3.5 in \cite{Th3}. Denote by $P : C \to A$ the canonical conditional expectation and note that $\tau_x\circ P$ is a trace on $C$. Since $\eta \in K_0(C)^+ \backslash \{0\}$ and $C$ is simple it follows that
$$
{(\tau_x \circ P)}_*(\eta) > 0 .
$$
Since 
\begin{align*}
&{(\tau_x \circ P)}_*(\eta) = g(x) = \hat{L}(\xi)(x) = \sum_{n \in \mathbb Z} L(\theta(\xi_n))(x) \\
&=  \sum_{n \in \mathbb Z}\theta(\xi_n)(x) = \theta( \Sigma(\xi))(x),
\end{align*} 
and $x \in \pi^{-1}(0)$ was arbitrary, it follows that $S(\eta) = \Sigma(\xi) \in H^+ \backslash \{0\}$. Hence, $S(K_0(C)^+) \subseteq H^+$.
\end{proof}

\begin{lemma}\label{02-09-21c} $K_1(C) = 0$.
\end{lemma}
\begin{proof} To establish this from the Pimsner-Voiculescu exact sequence, \cite{PV}, we must show that $\id -\rho$ is injective. Let $(\xi,g) \in \Gamma$ and assume that $\rho(\xi,g) = (\xi,g)$. Then $\sigma(\xi) = \xi$, implying that $\xi = 0$ and hence that $g|_{\pi^{-1}(0)} =0$. Since $(1-e^{-\pi(s)})g(s) = 0$ for all $s\in S$, it follows that $g =0$.
\end{proof}

Let $(S,\pi)$ and $B$ be as in Theorem \ref{main}. With $H = K_0(B)$ and the assumed identification of the tracial state space $T(B)$ of $B$ with $\pi^{-1}(0)$ we get the homomorphism $\theta : H \to \Aff \pi^{-1}(0)$ from the canonical map $K_0(B) \to \Aff T(B)$. It follows from Theorem 4.11 in \cite{GH} that $\theta(K_0(B))$ is dense in $\Aff \pi^{-1}(0)$ and that
$$
K_0(B)^+ = \left\{ h \in K_0(B): \ \theta(h)(x) > 0 \ \forall x \in \pi^{-1}(0)\right\} \cup \{0\} .
$$
We can therefore apply the preceding with $H = K_0(B), \ H^+ = K_0(B)^+$ and $u = [1]$.

Let $\tau$ be a trace state on $eCe$. Then $\tau_* \circ S^{-1} : H \to \mathbb R$ is a positive homomorphism such that 
$\tau_* \circ S^{-1}(u) = \tau_*(q(v)) = \tau(e) = 1$, and there is therefore a unique trace state 
$\tau'$ on $B$ such that 
\begin{equation}\label{17-07-23}
{\tau'}_* =  \tau_* \circ S^{-1}
\end{equation}
on $K_0(B) = H$.

\begin{lemma}\label{02-09-21b} The map $\tau \mapsto \tau'$ is an affine homeomorphism from $T(eCe)$ onto $T(B)$.
\end{lemma}
\begin{proof} The map is clearly affine. To show that it is continuous, assume that $\{\tau_n\}$ is a convergent sequence in $T(eCe)$ and let $\tau = \lim_{n \to \infty} \tau_n$. 
It follows from Lemma \ref{21-06-23x} that $eCe$ is a full corner in $C$, implying that the embedding $eCe \subseteq C$ induces an isomorphism on $K_0$. It follows therefore that
$$
\lim_{n \to \infty} {\tau_n}_* \circ S^{-1}(h) = \tau_*\circ S^{-1}(h)
$$ 
for all $h \in H$, and hence $\lim_{n \to \infty} \tau'_n = \tau'$ in $T(B)$ because $B$ is an AF-algebra. To see that the map is surjective, let $\tau \in T(B)$. There is then an $x \in \pi^{-1}(0)$ such that $\tau_*(h) = \theta(h)(x)$ for all $h$. By Corollary \ref{06-08-21a} applied with $\beta = 0$, there is an trace state $\omega \in T(eCe)$ such that $(\omega|_{eAe})_* = \ev_x$ on $K_0(eAe) \subseteq \Gamma$. Let $h \in H$. Then $(h^{(0)}, L(\theta(h))\psi_1^0\circ \pi )\in K_0(eAe)$ and $S_0((h^{(0)}, L(\theta(h))\psi_1^0\circ \pi )) = h$. Hence 
\begin{align*}
&\omega_* \circ S^{-1}(h) = (\omega|_{eAe})_*( q(h^{(0)}, L(\theta(h))\psi_1^0\circ \pi ))\\
& = L(\theta(h))(x) = \theta(h)(x) = \tau_*(h)
\end{align*}
for all $h\in H$, showing that $\omega' = \tau$. To see that the map is also injective, consider $\tau_1,\tau_2 \in T(eCe)$. If $\tau_1' = \tau_2'$, it follows that ${\tau_1}_* = {\tau_2}_*$ on $K_0(eCe)$. It follows from Lemma 4.6 in \cite{CP} that there are densely defined lower semi-continuous traces $\tilde{\tau}_i$ on $C$ such that $\tilde{\tau}_i|_{eCe} = \tau_i, \ i = 1,2$, and then $\tilde{\tau_1}_* = \tilde{\tau_2}_*$ on $K_0(C) = K_0(eCe)$. In particular, ${\tilde{\tau_1}}_*\circ \iota_* = {\tilde{\tau_2}}_*\circ \iota_*$ where $\iota : A \to C$ is the canonical embedding. Since $A$ is AF this implies that $\tilde{\tau_1}|_A = \tilde{\tau_2}|_A$. Thanks to (A) from Additional properties \ref{07-09-21c} it follows first that $\tilde{\tau_1} = \tilde{\tau_2}$ and then that $\tau_1 = \tau_2$.
\end{proof}

\begin{lemma}\label{02-09-21k} $eCe$ is $*$-isomorphic to $B$.
\end{lemma}
\begin{proof} Since $B$ is AF the $K_1$ group of $B$ is trivial, and by Lemma \ref{02-09-21c} the same is true for $eCe$ since $eCe$ is stably isomorphic to $C$. The affine homeomorphism $\tau \mapsto \tau'$ of Lemma  \ref{02-09-21b} is compatible with the isomorphism of ordered groups $S : K_0(eCe) \to K_0(B)$ from Lemma \ref{02-09-21} in the sense that ${\tau'}_* \circ S = \tau_*$, resulting in an isomorphism from the Elliott invariant of $eCe$ onto that of $B$. Both algebras, $B$ and $eCe$, are separable, simple, unital, nuclear and in the UCT class. It is well known that all infinite-dimensional unital simple AF-algebras are approximately divisible and hence $\mathcal Z$-absorbing by Theorem 2.3 of \cite{TW}; in particular, $B$ is $\mathcal Z$-absorbing. Since $C$ is $\mathcal Z$-absorbing thanks to (B) in Additional properties \ref{07-09-21c}, it follows from Corollary 3.2 of \cite{TW} that $eCe$ is $\mathcal Z$-absorbing. Therefore $eCe$ is isomorphic to $B$ by Theorem A in \cite{CGSTW}.
\end{proof} 
 
 In previous work, \cite{Th3}, \cite{ET} and \cite{EST}, we used the formulation of the classification theorem from Corollary D of \cite{CETWW}. The proof behind the statement in \cite{CETWW} is forbiddingly long and we have therefore chosen to refer to \cite{CGSTW} which presents a more direct proof. In addition, the formulation in \cite{CGSTW} makes it unnecessary to check that the isomorphism of $K_0$-groups preserves the ordering. As observed in \cite{CGSTW} this is automatic thanks to the $\mathcal Z$-stability of the two algebras and because the pairing maps are preserved under an isomorphism of the invariants. By inspection of the arguments given above this means that the second half of the proof of Lemma \ref{02-09-21} is unnecessary if one employs \cite{CGSTW} rather than \cite{CETWW}.%\footnote{Although it is not needed here, it is worth observing that the result in \cite{CGSTW} is also stronger than the statement in \cite{CETWW} because it is part of the formulation in \cite{CGSTW} that the isomorphism of algebras is a lift of the isomorphism of the invariants.}

\emph{Proof of Theorem \ref{main}:} Combine Lemma \ref{02-09-21k}, Lemma \ref{02-09-21d} and Proposition \ref{21-06-23}. 
\qed

\section{Concluding remarks}

\begin{itemize}

\item 
 We have chosen to show how the constructions from \cite{ET} can be changed to control the sets of $\KMS_\infty$ states and $\KMS_{-\infty}$ states when the given proper simplex bundle is unbounded in both directions. A similar modification of the methods from the proof of Theorem 5.1 in \cite{EST} is also possible, leading to a version of Theorem \ref{main} where $B$ is replaced by a Kirchberg algebra in the UCT class and with torsion free $K_1$-group. Similarly, it is also possible to consider the cases where the given proper simplex bundle is only unbounded in one direction. None of these additional cases present any new difficulties and can be handled by adopting the methods we have presented here. 

\item The work behind this paper started from a wish to decide if the set of $\KMS_\infty$ states of a flow has to be a convex set of states. The main result implies that this is not the case; a fact which seems to contradict Proposition 3.8 on page 447 of \cite{CM2}. The discrepancy arises from the different definitions of $\KMS_\infty$ states mentioned in the introduction. With the original definition of Connes and Marcolli the set is indeed a compact convex set of states, but it is always empty when the set of inverse temperatures does not contain an unbounded interval. Even when the set of inverse temperatures does contain an interval of the form $[r,\infty)$, it follows from the result we have obtained that the set of states satisfying Connes and Marcollis condition will generally be a proper subset of the set $\KMS_\infty(\sigma)$ as we have defined it here.

\end{itemize}

\end{document}